\documentclass[12pt]{amsart}

\title{Arithmetic geometry of character varieties with regular monodromy}

\author{Masoud Kamgarpour$^1$}
\author{GyeongHyeon Nam$^2$}
\author{Anna Pusk\'{a}s$^3$}

\address{$^1$School of Mathematics and Physics, The University of Queensland, QLD 4072, Australia}
\email{\href{mailto:masoud@uq.edu.au}{masoud@uq.edu.au}}

\address{$^2$Department of Mathematics, Ajou University, Suwon 16499, Republic of Korea}
\email{\href{mailto:ghnam@ajou.ac.kr}{ghnam@ajou.ac.kr}}

\address{$^3$School of Mathematics \& Statistics,  University of Glasgow,  G12 8QQ, Glasgow, Scotland}
\email{\href{mailto:anna.puskas@glasgow.ac.uk}{anna.puskas@glasgow.ac.uk}}

\usepackage{textgreek, makecell, enumerate, xcolor, bm, mathtools, thmtools, float, cite, wrapfig, multicol}

\usepackage{fullpage} 
 \usepackage[alphabetic]{amsrefs}
\usepackage{amssymb}
\usepackage{amsmath}
\usepackage{mathtools} 
\usepackage{xcolor}
\usepackage{colonequals}
\usepackage{amsrefs}
\usepackage[linktoc=all]{hyperref}
\usepackage[overload]{textcase}
\usepackage{textgreek}
\usepackage{tikz-cd}
\usepackage{mleftright}

\makeatletter
\newcommand{\vo}{\vec{o}\@ifnextchar{^}{\,}{}}
\makeatother

\theoremstyle{plain}
\newtheorem{thm}{Theorem}
\numberwithin{thm}{subsection} 
\newtheorem{lem}[thm]{Lemma}
\newtheorem{prop}[thm]{Proposition}

\newtheorem{cor}[thm]{Corollary}
\theoremstyle{definition}
\newtheorem{defe}[thm]{Definition}

\newtheorem{rem}[thm]{Remark}

\definecolor{red}{rgb}{1,0,0}
\definecolor{orange}{rgb}{1,0.5,0}
\definecolor{purple}{rgb}{.5,.2,.8}
\definecolor{blue}{rgb}{.2,.2,.8}
\definecolor{green}{rgb}{.4,.6,.4}

\newcommand{\ra}{\rightarrow}
\newcommand{\diag}{{\mathrm{diag}}}

\def\det{\mathrm{det}}

\def\bes{\begin{equation*}}  \def\ees{\end{equation*}} 
\def\bi{\begin{itemize}}   \def\ei{\end{itemize}}
\def\ba{\begin{eqnarray}} \def\ea{\end{eqnarray}}    
\def\bl{\begin{align}}    \def\el{\end{align}}       
\def\bls{\begin{align*}}    \def\els{\end{align*}}

\newcommand{\bC}{\mathbb{C}}
\newcommand{\bCt}{\mathbb{C}^\times}

\newcommand{\bX}{\mathbf{X}}

\newcommand{\GL}{\mathrm{GL}}
\newcommand{\bR}{\mathbf{R}}

\newcommand{\cB}{\mathcal{B}}

\newcommand{\F}{\mathbb{F}}

\newcommand{\PGL}{\mathrm{PGL}}

\newcommand{\Irr}{\mathrm{Irr}}

\newcommand{\Hom}{\mathrm{Hom}}

\newcommand{\cS}{\mathcal{S}}

\newcommand{\kf}{{k}}
\newcommand{\Fq}{\mathbb{F}_q}
\newcommand{\Fp}{\mathbb{F}_p}

\newcommand{\fg}{\mathfrak{g}}

\newcommand{\bGm}{\mathbb{G}_m}

\newcommand{\Imag}{{\mathrm{Im}}}
\newcommand{\Ad}{{\mathrm{Ad}}}

\newcommand{\rank}{{\mathrm{rank}}}
\newcommand{\fT}{\mathfrak{T}}

\newcommand{\Tor}{\mathrm{Tor}}


\begin{document}

\begin{abstract}  We count points on a family of smooth character varieties with regular semisimple and regular unipotent monodromies. We show that these varieties are polynomial count and obtain an explicit expression for their $E$-polynomials using complex representation theory of finite reductive groups. 
As an application, we give an example of a cohomologically rigid representation which is not physically rigid. 
\end{abstract}

\maketitle

\tableofcontents

\section{Introduction and the main results} 
\subsection{Overview}
Character varieties of surface groups play a central role in diverse areas of mathematics such as the non-abelian Hodge theory \cites{Simpson92, Simpson94} and the geometric Langlands program \cites{BeilinsonDrinfeld, BenZviNadler}. The study of the topology and geometry of character varieties has been a subject of active research for decades. 
In their groundbreaking works \cite{HRV, HLRV}, Hausel, Letellier, and Rodriguez-Villegas counted points on character varieties associated to  punctured surface groups and $\GL_n$. Their work has led to much further progress in understanding the arithmetic geometry of character varieties, cf. 
\cite{dCHM, Letellier, BaragliaHekmati, Schiffmann, Mellit, LetellierRV, Ballandras, BridgerKamgarpour}.

Most prior research in this area has focused on character varieties for groups of type $A$. The only exceptions we are aware of are \cite{Cambo}, which investigates $\mathrm{Sp}_{2n}$ character varieties, and \cite{BridgerKamgarpour}, which explores character stacks for general reductive groups and compact surfaces. Consequently, our understanding of character varieties for general reductive groups remains quite limited compared to the extensive knowledge in type $A$. This situation parallels the state of character theory for finite groups of Lie type in the 1960s, before the work of Deligne and Lusztig \cite{DeligneLusztig}. It is important to note that for many applications, such as those in the Langlands program or mirror symmetry, a deeper understanding of character varieties for more general groups is essential. This paper aims to take the first step towards extending the program of Hausel, Letellier, and Rodriguez-Villegas beyond type $A$ to encompass more general groups.

In this work, we study character varieties associated with arbitrary reductive groups, assuming both regular semisimple and regular unipotent monodromies at the punctures. We count points on these varieties over $\Fq$ and demonstrate that the point count is given by a polynomial in $q$.\footnote{In a companion paper \cite{Bailey}, character varieties associated to generic semisimple classes are considered.} This polynomial, known as the ``$E$-polynomial" of the character variety, encodes significant information, including the number of connected components and the Euler characteristic. For further details on $E$-polynomials and their applications to character varieties, we refer the reader to the appendix of \cite{HLRV} or \cite[\S 2.2]{LetellierRV}. A notable feature of our approach is that it applies uniformly across all reductive groups $G$, avoiding a case-by-case analysis based on group type.

\subsection{Definition of character varieties} \label{ss:def}
   Fix non-negative integers $g$ and $n$. Let $\Gamma=\Gamma_{g,n}$ be the fundamental group of a compact orientable surface with genus $g$ and $n$ punctures; i.e., 
\[
\Gamma = \frac{\langle a_1, b_1, \ldots , a_g, b_g, c_1, \ldots , c_n\rangle}{[a_1,b_1] \cdots [a_g , b_g ]c_1 \cdots c_n}.
\] 
Let $G$ be a connected split reductive group over a field $k$ and $(C_1,\ldots ,C_n)$ an $n$-tuple of conjugacy classes of $G$. 
The associated \emph{representation variety} $\mathbf{R}=\mathbf{R}(C_1,...,C_n)$   is 
\[
\mathbf{R} \colonequals 
\left\lbrace \left.(A_1, B_1,\ldots ,A_g , B_g, S_1 ,\ldots , S_n) \in G^{2g}\times \prod_{i=1}^n C_{i} \, \right|\, [A_1,B_1]\cdots [A_g ,B_g] S_1 \cdots S_n = 1 \right\rbrace.
\]

This is an affine scheme of finite type over $k$. The group $G$ acts on $\mathbf{R}$ by conjugation and the corresponding GIT quotient
\[
\bX:=\bR/\!\!/G=\bR/\!\!/(G/Z)
\] 
(where $Z$ is the centre of $G$) is called the \emph{character variety} associated to $\Gamma$, $G$, and $(C_1,...,C_n)$. We also have the associated quotient stack $[\bX]:=[\bR/(G/Z)]$.

\subsubsection{Non-emptiness}  \label{s:emptiness}
Note that $\bR(k)$ (and therefore $\bX(k)$) is empty, unless there exist $S_i\in C_i(k)$ such that the product $S_1\cdots S_n$ is in $[G(k), G(k)]$. This is equivalent to requiring that \emph{for every} $S_i\in C_i(k)$ the product $S_1\cdots S_n$ is in $[G(k), G(k)]$.  Henceforth, we assume that this condition is satisfied. If $g>0$ then this condition is also sufficient for $\bR$ to be non-empty, because the commutator map $G\times G\ra [G,G]$ is surjective, cf. \cite{Ree}. 
In contrast, when $g=0$, deciding if $\bR$ is non-empty is subtle and is closely related to the Deligne--Simpson problem, cf. \cite{Simpson, Kostov}. We refer the reader to Theorem \ref{t:topology} for some new results regarding non-emptiness.

\subsection{Smoothness}  Character varieties are often singular.  Our first main result is a criteria for smoothness. We assume throughout that characteristic of $k$ is a prime which is \emph{very good} for $G$. This means that $p\neq 2$ in types $B, C, D$, $p\neq 2, 3$ in types $G_2, F_4, E_6, E_7$, $p\neq 2, 3, 5$ in $E_8$, and $p\nmid n+1$ for type $A_n$.

\begin{thm}\label{t:Smooth} 
Suppose $\bR$ is non-empty and $G/Z$ acts freely on it. Then 
\begin{enumerate} 
\item[(i)] $\bR$ is smooth of pure dimension
\[
\displaystyle \dim(\bR)=2g\dim(G)-\dim([G,G])+\sum_{i=1}^n\dim(C_i).
\]   
\item[(ii)] The canonical map $[\bX]\ra \bX$ from the quotient stack into the GIT quotient is an isomorphism. Thus, $\bR$ is a principal $G/Z$-bundle over $\bX$ in the \'{e}tale topology and $\bX$ is  smooth of pure dimension 
\[
\dim(\bX)=2g\dim(G)-2\dim([G,G])+\sum_{i=1}^n\dim(C_i).
\] 
\end{enumerate} 
\end{thm} 

Over complex numbers, theorems of this kind are by now standard, cf. \cite[Theorem 2.1.5]{HLRV} or \cite[\S 9.3]{Boalch}. However, some care is required when working with  arbitrary $G$ in positive characteristics. For completeness, we provide a proof in \S \ref{s:Smooth}.

 \subsubsection{Free action} 
We now consider a case where it is easy to see that $G/Z$ acts freely on $\bR$. Let $T$ be a maximal split torus of $G$. Following Steinberg \cite{Steinberg}, 
we call a semisimple element $S\in T$ \emph{strongly regular} if $C_G(S)=T$.

\begin{lem}\label{l:free}  Assume one of the $C_i$'s is strongly regular semisimple and another one is regular unipotent. Then the action of $G/Z$ on $\bR$ is free. 
\end{lem} 

\begin{proof} The centraliser $C_G(N)$ of a regular unipotent element is a product $ZA$ of the centre and a (possibly disconnected) unipotent group $A$, cf.  Theorem 4.11 and 4.12 of \cite{Springer}. Thus, if $S\in T$ is a strongly regular element, then 
\[
C_G(S)\cap C_G(N) = T\cap ZA = Z. 
\]
Here, the last equality follows from the fact that the only semisimple elements of $ZA$ are the central ones. 
\end{proof}

 \begin{cor}\label{c:free} Under the assumptions of the lemma, $\bR$ and $\bX$ are smooth.
 \end{cor}

Note that under these assumptions, every representation in $\bR$ is stable under the action of $G/Z$ (in the sense of geometric invariant theory). However, such a representation is  not necessarily irreducible, see \S \ref{s:GL} for some examples. This kind of phenomenon is rare. Stable representations are often irreducible.

\subsection{Character varieties with regular monodromy} We  now define the main object of our study and state the  main results. 

\begin{defe} \label{d:main} 
Let $m$ and $n$ be integers satisfying $1\leq m<n$. Let $(C_1,...,C_n)$ be conjugacy classes satisfying the following properties: 
\begin{itemize} 
\item $C_1,...,C_m$ are classes of strongly regular elements $S_1,...,S_m\in T(k)$; 
\item $C_{m+1},...,C_n$ are regular unipotent classes; 
\item The product $S_1\cdots S_n$ is in $[G(k), G(k)]$. 
\end{itemize} 
We call the resulting character variety $\bX=\bX(C_1,...,C_n)$ a \emph{character variety with regular monodromy}.
\end{defe} 

In considering these character varieties, we were motivated by the work of Deligne and Flicker \cite{deligne2013counting}, who counted local systems with regular unipotent monodromy in the $\ell$-adic setting. In our case, the inclusion of regular semisimple monodromy is convenient because it implies that the action of $G/Z$ is free. In addition, it makes counting points over finite fields simpler.

\subsubsection{Polynomial property}\label{sssect:polyproperty}  For the rest of the introduction, we assume that $k=\Fq$.  Let us call a variety $Y$ over $k$ polynomial count if there exists a polynomial $|\!|Y|\!|\in \mathbb{Z}[t]$ such that  
\[
|Y(\mathbb{F}_{q^{n}})|=|\!|Y|\!|(q^{n}),\qquad\quad  \forall n\geq 1. 
\]
In this case, $|\!|Y|\!|$ is called the counting polynomial of $Y$. 

\begin{thm} \label{t:count} Suppose $G$ has connected centre. Then after a finite base change (i.e., replacing $\mathbb{F}_q$ with $\mathbb{F}_{q^\ell}$ for an appropriate $\ell$), the character variety $\bX$ with regular monodromy (Definition \ref{d:main}) is polynomial count. 
\end{thm}

Here,  by finite base change we mean replacing $k=\Fq$ by an extension $\F_{q^\ell}$. We refer the reader to \S \ref{s:countPrecise} for a discussion about how large this extension needs to be and the (rather complicated) explicit form of the counting polynomial. Examples of this counting polynomial for groups of low rank are given in \S \ref{s:Examples}.

 The first observation in the proof of the above theorem is that by Corollary \ref{c:free}, $G/Z$ is acting freely on $\bR$. Thus, the character variety $\bX$ agrees with the character stack $[\bX]$. 
Next, a formula that goes back to the work Frobenius and Mednykh states that
\begin{equation}\label{eq:Frob}
 \displaystyle |[\bX](k)|=  |Z(k)|  \sum_{\chi\in \Irr(G(k))} \left(\frac{|G(k)|}{\chi(1)}\right)^{2g-2}\prod_{i=1}^n \frac{\chi(C_i(k))}{\chi(1)} |C_i(k)|. 
\end{equation} 
For a very similar version of this mass formula, see \cite[(1.2.2)]{HLRV}. The classical version of this formula, counting $|\Hom(\pi_1(\Sigma), G)|$ with $G$ a finite group and $\Sigma $ a surface with punctures, is attributed to Frobenius (or Frobenius--Schur) in the genus zero case and Mednykh in the case without punctures, cf. \cite{mednykh}. 

Here, $\Irr(G(k))$ denotes the set of irreducible complex characters of the finite reductive group $G(k)$. Using Deligne--Lusztig theory, we reduce the evaluation of $|\bX(k)|$ to computing a certain sum of characters of the finite abelian group $T(k)$. This is Part I of the proof and is carried out in \S \ref{s:PartOne}. Part II concerns evaluating these character sums using Pontryagin duality and M\"obius inversion on the partially ordered set of closed subsystems of the root system of the Langlands dual group $G^\vee$. See \S\ref{s:PartTwo} for details of the proof.

\subsubsection{Remarks} 
\begin{enumerate} 
\item  The above theorem is new even for $G=\GL_n$, though we were informed by Letellier that in this case, it can be extracted from results of \cite{Letellier}. 

\item We expect the theorem holds if $Z(G)$ is disconnected; however, the analysis becomes more intricate due to the familiar pitfall that centralisers of semisimple elements in $G^\vee$ need not be connected. 

\end{enumerate}

As noted in the overview section, whenever the number of points of a variety over a finite field is a polynomial, this polynomial gives important topological implication for the variety. 
 We now discuss the topological implications of Theorem \ref{t:count}. We assume that a finite base change is already performed so that $\bX$ is polynomial count. Note that if $G$ is commutative, then $\bX=G^{2g}$ (see \S \ref{s:Torus}); thus, we may assume $G$ is non-commutative.

 \begin{thm}\label{t:topology}  Suppose $Z(G)$ is connected and  $G$ is non-commutative.
\begin{enumerate} 
\item[(i)]  Suppose either $g>0$ or $n>3$. Then $|\pi_0(\bX)|=|\pi_1([G,G])|$; in particular, $\bX$ is non-empty. 
\item[(ii)] Suppose either $g>0$ or $n>m+2$. Then the Euler characteristic of $\bX$ is $0$. 
\end{enumerate} 
\end{thm} 

We prove this theorem by analysing the counting polynomial of $\bX$; see  \S \ref{s:topology} for details. 

\subsubsection{Remarks} 

\begin{enumerate} 
\item  If $(g,n)=(0,2)$ then $\bR$ and $\bX$ are empty. (This is because $1\notin C_1C_2$, cf. Definition \ref{d:main}.) The case $(g,n)=(0,3)$ exhibits interesting properties.  For instance, if $G=\GL_2$ and if we have two regular semisimple conjugacy classes, then $\bX$ can either be a single point or two points. See \S \ref{s:subtle2} for details. 
\item In \S \ref{s:subtle} we consider the case where $G$ equals $\GL_2$ or $\GL_3$ and $(g,n,m)$ equal to $(0,4,3)$ or $(0,4,2)$. We show that in these cases, the Euler characteristic may be  non-zero. 
The fact that the Euler characteristic behaves differently in genus $0$ is also observed in \cite[Remark $5.3.4$]{HLRV}.  
   \end{enumerate}

\subsection{Applications to rigidity} We use the notation of \S \ref{ss:def} and assume $g=0$. Let $\rho: \Gamma \ra G$ be an element of $\bR=\bR(\Gamma, G, C)\subseteq \Hom(\Gamma, G)$. Recall the following definitions: 
\begin{enumerate} 
\item  $\rho$ is  \emph{irreducible} if $\rho(\Gamma)$ is not contained in any proper parabolic subgroup of $G$. 
\item $\rho$ is  \emph{cohomologically rigid} (\cite[\S 7]{FG}) if it is irreducible and 
\begin{equation}\label{eq:rigidity} 
\sum_{i=1}^n \dim(C_i) = 2\dim([G,G]).
\end{equation} 
\item $\rho$ is \emph{physically rigid} if it is irreducible and determined, up to isomorphism, by the conjugacy classes $C_1,\dots, C_n$. (This is equivalent to requiring that every element of $\bR$ is isomorphic to $\rho$). 
\end{enumerate} 

For $G=\GL_n$, Katz has proven that physical and cohomological rigidity are equivalent for elements of $\Hom(\Gamma, \GL_n)$ \cite[\S 1]{KatzRigid}. It has been expected that these two notions are not equivalent for $G\neq \GL_n$, but no (counter)examples were known. Theorem \ref{t:topology} leads to an example of a cohomologically rigid representation which is not physically rigid. 

This example is constructed explicitly in \S \ref{ss:rigid} by considering the character variety with regular monodromy associated to $(G,g,n)=(\PGL_2, 0, 3)$. Using Theorem \ref{t:count}, one verifies that $|\bX|=2$.  By choosing eigenvalues carefully, one can show that elements of $\bR$ are irreducible.  In view of Theorem \ref{t:Smooth}, the fact that $\dim(\bX)$ is zero implies that
\eqref{eq:rigidity} is satisfied, so every element of $\bR$ is cohomologically rigid. But the fact that $| \bX | = 2$ implies that there exist two representations in $\bR$ which are not isomorphic. We conclude that representations in $\bR$ are cohomologically but not physically rigid.

\subsection{Notation}  \label{ss:notation} 
Throughout the paper, $k$ denotes a field of characteristic $p$,  $G$ a connected split reductive group over $k$, $Z=Z(G)$ the centre of $G$, $T$  a maximal split torus of $G$,  $B$ a Borel subgroup containing $T$, $U$ the unipotent radical of $B$, $W$ the Weyl group, $r:=\dim(T\cap [G,G])$ the semisimple rank, $(X,  \Phi, X^\vee, \Phi^\vee)$ the root datum,  $\langle \Phi \rangle$ the root lattice, and $\langle \Phi^\vee \rangle$  the coroot lattice. We assume that $Z(G)$ is connected and that the characteristic of $k$ is very good for $G$. 

\subsubsection{} 
For each root subsystem $\Psi \subseteq \Phi^\vee$ (including the empty one), we have a reflection  subgroup $W(\Psi)\subseteq W$ generated by reflections associated to roots $\alpha\in \Psi$. Let $\mathcal{S}(\Phi^\vee)$ denote the set of closed subsystems of $\Phi^\vee$. (Recall a subset $\Psi \subseteq \Phi^\vee$ is a closed subsystem if it is itself a root system, and $\alpha,\beta\in \Psi ,$ $\alpha+\beta \in \Phi^\vee$ implies $\alpha+\beta \in \Psi $.) This is a partially ordered set (poset) with the ordering given by inclusion. The M\"obius function associated to this poset plays a crucial role in our point count. 

\subsubsection{} 
Let $T^\vee:=\mathrm{Spec}\, k[X]$ be the dual torus and $G^\vee$ the Langlands dual group of $G$ over $k$. In other words, $G^\vee$ is a connected split reductive group over $k$ with maximal split torus $T^\vee$ and root datum $(X^\vee, X, \Phi^\vee, \Phi)$. For a finite abelian group $A$, we let $A^\vee:=\Hom(A, \bCt)$ denote its Pontryagin dual. When $k$ is a finite field, we identify the Pontryagin dual $T(k)^\vee$ with the Langlands dual $T^\vee(k)$ as follows.  

\subsubsection{}\label{sss:dualTorus} 
Let $\mu_{\infty,p'}(\bC)$ denote the set of  roots of unity in $\bCt$ whose order is prime to $p$. We choose, once and for all, two isomorphisms: 
\begin{equation}\label{eq:rootsOfUnity}
 \overline{\Fp}^\times \simeq (\mathbb{Q}/\mathbb{Z})_{p'} \qquad \textrm{and} \qquad (\mathbb{Q}/\mathbb{Z})_{p'}\simeq \mu_{\infty,p'}(\bC). 
\end{equation} 
 As noted in \cite[\S 5]{DeligneLusztig}, this induces an isomorphism  between the  Pontryagin dual and the Langlands dual: 
\[
T(k)^\vee\simeq T^{\vee}(k).
\]

\subsection{Acknowledgements} 
 We would like to thank David Baraglia, Phillip Boalch, Nick Bridger, Valentin Buciumas, Jack Hall,  Konstantin Jakob, Emmanuel Letellier, Paul Levy, Peter McNamara, Dinakar Muthiah, Arun Ram, Ryan Vinroot, Matthew Spong,  Ole Warnaar, Geordie Williamson, and Yang Zhang for helpful conversations. 
 
MK was supported by Australian Research Council Discovery Project DP200102316. GN was supported by an Australian Government Postgraduate Award and the National Research Foundation of Korea (NRF) grant funded by the Korea government (MSIT) (No. RS-2024-00334558). AP was supported by an Australian Research Council Discovery Early Career Research Award DE200101802.

\section{Smoothness of character varieties} \label{s:Smooth} The goal of this section is to prove Theorem \ref{t:Smooth}. Our proof of Part (i) of the theorem is modelled on \cite[Theorem 2.1.5]{HLRV}. The  novelty here is that we work in the reductive setting and avoid using matrices.  We first give the proof in characteristic $0$ and then explain the modifications required in positive characteristics.

\subsection{Proof of Part (i) in characteristic $0$}\label{ss:char0}   We start with some notation. 

\subsubsection{} 
 For each $h\in G$, let $l_h: G\ra G$ denote the left multiplication map. This is an isomorphism of varieties; thus, it induces an isomorphism of vector spaces 
\[
dl_h: T_{g} G \xrightarrow{\sim} T_{hg} G.
\] 
Similarly, we have the right multiplication map $r_h$  and its derivative 
\[
dr_h: T_g G \xrightarrow{\sim} T_{gh} G.
\] 
Thus, 
$ dr_{h^{-1}}\circ dl_{h}$ is a linear map $T_g G \rightarrow T_{hgh^{-1}} G$. 
If $g$ is the identity of $G$, then $T_gG=:\fg$ is the Lie algebra of $G$ and this is a linear automorphism $\fg\ra \fg$ which equals the adjoint map $\Ad_h$.

\subsubsection{} 
Let
\[
F: G^{2g}\times \prod_{i=1}^n C_{_i}\ra [G,G]
\]
be the morphism of varieties 
defined by
  \[
F(A_1, B_1, \ldots  ,A_g , B_g, S_1 , \ldots , S_n):= [A_1,B_1] \cdots  [A_g ,B_g] S_1 \cdots  S_n. 
\]
Note that our assumption in \S \ref{s:emptiness} implies that the image of $F$ is indeed in $[G,G]$. 
By definition, the representation variety is $\bR=F^{-1}(1)$.

\subsubsection{}  Now consider a point 
\[
r=(A_1, B_1, \ldots  ,A_g , B_g, S_1 , \ldots , S_n)\in \bR=F^{-1}(1).
\] 
It follows from the Regular Value Theorem (see, e.g., \S 25 of \cite{Ravi}) that $\bR$ is smooth and equidimensional if the differential 
\[
dF_r: T_r(G^{2g}\times \prod_{i=1}^n C_{_i}) \ra T_1([G,G])=[\fg,\fg]
\]
 is surjective.   To establish surjectivity, we construct an auxiliary surjective map 
 \[
 \phi=\phi_r: \fg^{2g+n}\ra [\fg,\fg]
 \]
  and prove that the image of $dF_r$ equals the image of $\phi$. To this end, we need yet another auxiliary map $\Psi=\Psi_r$.

\subsubsection{} 
Consider the morphism of varieties 
\[
\Psi=\Psi_r:  G^{2g+n} \ra G^{2g}\times \prod_{i=1}^n C_i,
\]
defined by 
\[
\Psi(x_1, y_1,...,x_{g}, y_{g}, z_1,...,z_n):=(A_1 x_1, B_1 y_1 , ..., A_g x_g, B_g y_g, z_1^{-1} S_1 z_1,..., z_{n}^{-1} S_n z_n).
\]
In other words, 
\[
\Psi = (l_{A_1},l_{B_1},\ldots l_{A_g},l_{B_g},\eta_1,\ldots ,\eta_n), 
\]
 where $\eta_j:G\rightarrow C_j$ is given by $\eta_j(g)=g^{-1}S_jg$. 

\subsubsection{} 
Observe that $\Psi(1)=r$. Thus, taking the differential at $1$, we obtain a linear map  
\[
d\Psi_1: \fg^{2g+n} \ra T_r\left(G^{2g}\times \prod_{i=1}^n C_{_i}\right)=\prod_{i=1}^g(T_{A_i}G\times T_{B_i}G)\times \prod_{j=1}^{n}T_{S_j}C_j.
\]
The auxiliary function $\phi$ is defined as follows:  
\[
\phi: \fg^{2g+n}\ra [\fg, \fg], \qquad \phi:= dF_r \circ d\Psi_1.
\]
To prove that $\phi$ has the desired properties, we first need some notation.

\subsubsection{}  For $i\in \{1,2,...,g\}$ and $j\in \{1,2,...,n\}$, let  
\[
X_i\in T_{A_i}G,\qquad Y_i\in T_{B_i}G,\qquad \textrm{and} \qquad Z_j\in T_{S_j}C_j\subseteq T_{S_j}G.
\]
Let $X_i'$, $Y_i'$, and $Z_j'$ be elements of $\fg$ satisfying the following:
\[
X_i'=(dl_{A_i})^{-1}(X_i),\qquad Y_i'=(dl_{B_i})^{-1}(Y_i),\qquad -Z_j'\in (1-\Ad_{S_j})^{-1}((dr_{S_j})^{-1}(Z_j)). 
\]
Note that $dl_{A_i}$ and $dl_{B_i}$ are invertible and $d\eta_j=dr_{S_j}\circ (\Ad_{S_j}-1)=dl_{S_j}-dr_{S_j}:\fg\rightarrow T_{S_j}C_j$ is surjective \cite[\S 2]{richardson1967conjugacy}. Thus, the elements $X_i', Y_i', Z_j'$ do exist.

In addition, set 
\[
A_i \sharp B_i := [A_1,B_1]\cdots [A_{i},B_{i}]
\]
and 
\[
\overline{S_{j}}:=\prod_{t=j}^{n}S_t^{-1}. 
\]
For convenience, set $\overline{S_{n+1}}:=1$.

\begin{prop}\label{p:derivativeform_phi}
 We have 
\[
dF_r(X_1,Y_1,\ldots ,X_g,Y_g,Z_1,\ldots,Z_n) = \phi (X_1',Y_1',\ldots ,X_g',Y_g',Z_1',\ldots ,Z_n').
\]
Moreover, this equals  
\[
\sum_{i=1}^g \Ad_{(A_{i-1} \sharp B_{i-1})A_i}(\left(1-\Ad_{B_i}\right) X_i')+ 
\sum_{i=1}^g \Ad_{(A_{i-1} \sharp B_{i-1})A_iB_i}((1-\Ad_{A_i^{-1}}) Y_i')+
\sum_{j=1}^{n} \Ad_{\overline{S_{j+1}}}((1-\Ad_{S_j})(Z_j')).
\] 
\end{prop}

\begin{proof} This follows from repeated, straightforward, but quite tedious, application of the chain rule. 
\end{proof}

The first statement of the proposition implies that $\mathrm{Im}(dF_r)=\mathrm{Im}(\phi)$. (Alternatively, this follows from the definition of $\phi$ and the surjectivity of $d\Psi _1$.) We now use the second statement to show that $\phi$ is surjective. 

\subsubsection{} Let  $H$ be the closed subgroup of $G$ generated by the elements $A_1,B_1,\dots, A_g, B_g,$  $S_1,...,S_n$. We then have 
\begin{equation} \label{eq:centraliser}
Z(G)=C_{G}(H)\implies \mathrm{Lie}(Z(G))=\mathrm{Lie}(C_{G}(H)) = C_{\fg}(H)\implies [\fg,\fg]\cap C_{\fg}(H)=0.
\end{equation} 
Here $C_{\fg}(H)$ is defined as follows: 
\[
C_{\fg}(H):=\{x\in \fg \, | \, \Ad_h(x)=x,\, \, \forall h\in H\}.
\]
As the Killing form  $K=K_{[\fg,\fg]}$ is invariant and non-degenerate, we conclude that for every $h\in G$
\begin{equation}\label{eq:K-Ad-annih}
K(t,(1-\Ad_h) v)=0\quad \forall v \in [\mathfrak{g},\fg] \quad \implies \quad \Ad_ht=t.
\end{equation}
To see this, observe that
\[
K(t,(1-\Ad_h) v) = K(t,v)-K(t, \Ad_h v)= K(t,v)-K(\Ad_{h^{-1}} t, v)=K(t-\Ad_{h^{-1}} t, v).
\]
Returning to the map $\phi$, setting all but one of $X_i'$, $Y_i'$, and $Z_j'$ to zero, and repeatedly applying  observation \eqref{eq:K-Ad-annih}, it follows that if $t\in [\fg,\fg]$ such that $K(t,\Imag \phi )=0,$ then $t$ is fixed by every one of $\Ad_{A_i}$, $\Ad_{B_i}$, and $\Ad_{S_j}$. This implies $t\in C_{\fg}(H)$ and thus $t=0$. It follows that $\Imag(\phi)=[\fg,\fg]$, as required. 
 \qed

\subsection{Proof of Part (i) in positive characteristics} \label{s:positive} In positive characteristics, the notion of Lie algebra can be ambiguous so let us explain what we mean by this. The group $G$, being connected split reductive, has a canonical $\mathbb{Z}$-model $\mathtt{G}$. Let $\mathtt{g}:=\mathrm{Lie}(\mathtt{G})$ denote the Lie ring scheme of $\mathtt{G}$ over $\mathbb{Z}$. Then the Lie algebra $\fg$ is, by definition, the base change of $\mathtt{g}$ to $k$.

\subsubsection{}
There are three parts of the above proof that require care in positive characteristic: 
\begin{enumerate} 
\item[(i)] 
The equality  $C_{\mathfrak{g}}(H)=\mathrm{Lie}(C_G(H))$ does not always hold. This is the issue of ``separability" in positive characteristics. However, if we assume that $p$ is a very good prime for $G$  then this equality does hold, cf. \cite{BMRT}.
 \item[(ii)] 
The Killing form on $[\fg, \fg]$ may be degenerate. However, if we assume $p$ is very good for $G$, then a non-degenerate $G$-invariant bilinear form on $\fg$ exists, cf. \cite{Letellier}, and the proof goes through if we use this invariant form. 

 \item[(iii)] 
The map $d\eta _j:\fg\rightarrow T_{S_j}C_j$ is not necessarily surjective \cite[Lemma 2.1]{richardson1967conjugacy}. However, this is actually not needed. The first statement of Proposition \ref{p:derivativeform_phi} holds as long as $Z_j\in \Imag(d\eta_j).$ This implies $\Imag (dF_r)\supseteq \Imag (\phi )$ which is sufficient for the proof. 
\end{enumerate} 

The upshot is that the theorem and its proof are valid, provided we assume that $p$ is very good for $G$.

\subsection{General facts about actions of group schemes}  To prove part (ii) of the theorem, we need some facts about actions of group schemes. These facts are well-known to the experts but we could not find an appropriate reference in the literature. We momentarily use a more general notation than the rest of the paper. 
So let  $G$ be a group scheme acting on a scheme $X$.  Let $\varphi: G\times X\ra X\times X$ be the morphism 
\[
\varphi(g,x):= (g.x, x).
\]
Following \cite{Mumford}, we say that the action is  
 \begin{itemize} 
 \item    \emph{proper} if $\varphi$ is proper; 
 \item  \emph{free} if for every scheme $S$, the action of $G(S)$ on $X(S)$ is free; equivalently,  $\varphi$ is a monomorphism. 
 \item   \emph{scheme-theoretically free} if $\varphi$ is a closed immersion; i.e., a proper monomorphism.  
 \end{itemize}

\subsubsection{Example}
 The action of $\mathbb{G}_m$ on $\mathbb{A}^2-\{0\}$ given by $g.(x,y)=(gx, g^{-1}y)$ is free but not scheme-theoretically free. The GIT quotient is the affine line with a double point which is not separated.

 \subsubsection{} By definition, a free action is scheme-theoretically free if and only if it is proper.  In favourable situations, properness is automatic:
 
 \begin{prop}  Let $G$ be a connected reductive group over a field $k$ acting on an affine  scheme $X$ of finite type over $k$. Suppose the geometric points of $X$ have finite stabilisers under the action of $G$. Then, the action is proper. 
 \end{prop}

\begin{proof}  This is \cite[Proposition 0.8]{Mumford}. 
\end{proof} 

\begin{cor} Let $G$ be a connected reductive group over a field $k$ acting on an affine scheme $X$ of finite type over $k$. 
If the action is free, then it is scheme-theoretically free. 
\end{cor} 

\begin{proof} This is immediate from the previous proposition. 
\end{proof}

  \subsubsection{Quotients} Let  $G$ be a group scheme of finite type over a field $k$ acting on an affine scheme $X$ of finite type over $k$. 
 \begin{enumerate} 
 \item[(i)] The action of $G$ on $X$ is free if and only if  the quotient stack $[X/G]$ is represented by an algebraic space. This follows from the fact that an Artin stack is Deligne--Mumford if and only if it has an unramified diagonal, cf. Stacks Project, \S 100.21.
 \item[(ii)] If $G$ is connected reductive and the action of $G$ on $X$ is free then the canonical map  $[X/G]\ra X/\!\!/G$ is an isomorphism.  This is because, under our assumptions, $[X/G]$ is a separated algebraic space, and $X/\!\!/G$ is a categorical quotient in the category of separated algebraic spaces, cf. Theorem 7.2.1 of \cite{Alper}. 
 \end{enumerate}

\subsection{Proof of  part (ii)} By Lemma \ref{l:free}, the action of $G/Z$ on the representation variety $\bR$ is free. 
Since $\bR$ is affine and of finite type over $k$ and $G$ is connected reductive, the above discussions imply that the action is scheme-theoretically free and the canonical map $[\bX]=[\bR/(G/Z)]\ra  \bR/\!\!/(G/Z)=\bX$ is an isomorphism. This proves the first statement of (ii). The remaining statements follow immediately. \qed

    \section{Character varieties over finite fields} \label{s:PartOne} 
    In this section, we work over a finite field $k=\Fq$ and assume that the reductive group $G$ has connected centre and that $p=\mathrm{char}(k)$ is a very good prime for $G$. Let $\bX$ be a character variety with regular monodromy over $k$ (Definition \ref{d:main}). 
In what follows, we use the Frobenius Mass Formula \eqref{eq:Frob} and the character theory of finite reductive groups (\`{a} la Deligne--Lusztig) to obtain an expression for $|\bX(k)|$ involving certain sums of characters of $T(k)$. The evaluation of these character sums will be carried out in subsequent sections.

 \subsection{Recollections on characters of $G(k)$} Let $\Irr(G(k))$ denote the set of irreducible complex characters of $G(k)$.  We identify a representation with its character and denote them by the same letter.
 
 \subsubsection{} 
To every character $\theta \in T(k)^\vee$, one associates a  representation:  
\[
\cB(\theta)\colonequals \mathrm{Ind}_{B(k)}^{G(k)} \, \theta. 
\] 
Irreducible constituents of $\cB(\theta)$ are known as principal series representations. 

\begin{lem} \label{l:principalSeries}
If $\theta$ and $\theta'$ are $W$-conjugate then $\cB(\theta)$ and $\cB(\theta')$ are isomorphic. Otherwise, they have no isomorphic irreducible constituents. 
\end{lem}

\begin{proof} This is well-known, cf. \cite[Corollary 6.3]{DeligneLusztig}. 
\end{proof} 

Let $\langle .\, , .\rangle$ denote the standard invariant inner product on the class functions on $G(\Fq)$.
\begin{thm}[Deligne--Lusztig] \label{t:DL} Let $S\in T(k)$ be a strongly regular element (i.e., $C_G(S)=T$) and $\chi\in \Irr(G(k))$. Then 
\[
\chi(S)= \sum_{\theta \in T(k)^\vee} \langle \chi, \cB(\theta)\rangle\,  \theta(S). 
\]
In particular, $\chi(S)$ is non-zero only if $\chi$ is a principal series representation. 
\end{thm} 

\begin{proof} This is a corollary of Deligne and Lusztig's construction of representations of $G(k)$, see \cite[Corollary 7.6.2]{DeligneLusztig}. 
\end{proof}

\subsubsection{} \label{ss:chiTheta}
Our assumptions on $G$ and $p$ imply that regular unipotent elements of $G(k)$ form a single conjugacy class, cf. \cite[Theorem 4.14]{Springer}. Let $N\in B(k)$ be a regular unipotent element. 

\begin{thm}[Green--Lehrer--Lusztig]\label{t:GLL} 
\begin{enumerate} 
\item[(i)] For each $\theta\in T(k)^{\vee}$, the representation 
	$\cB(\theta)$ has a unique irreducible constituent that does not vanish on $N$. 
	\item[(ii)] Every irreducible character of $G$ takes value $0$ or $\pm 1$ on $N$. 
	\end{enumerate} 
\end{thm}

\begin{proof} See \cite{GreenLehrerLusztig} or \cite[Corollary 10.8]{DeligneLusztig}. 
\end{proof} 

We denote by $\chi_\theta$ the constituent of $\cB(\theta)$ which does not vanish on $N$. This is the unique semisimple constituent of $\cB(\theta)$, cf. \cite[\S 2.6.9]{geck2020character}.

\subsubsection{} \label{s:Steinberg} To study the representation $\chi_\theta$, we need some more information about the stabiliser of $\theta$. 
 For each $\theta \in T(k)^\vee$, let 
\[
W_\theta:=\{w\in W\, | \, w.\theta=\theta\}.
\]
In view of the identification in \S \ref{sss:dualTorus}, we can think of $\theta$ as an element of the dual torus $T^\vee(k)$; thus, we can define 
\[
\Phi_\theta^\vee:=\{\alpha \in \Phi^\vee \, | \, \alpha(\theta)=1\}.
\]

\begin{thm} 
\begin{enumerate} 
\item[(i)]  $\Phi_\theta^\vee$ is a closed subsystem of $\Phi^\vee$.
\item[(ii)] $W_\theta=W(\Phi_\theta^\vee)$; in particular, $W_\theta$ is a reflection subgroup. 
\end{enumerate} 
\label{t:Steinberg}
\end{thm} 

\begin{proof} For (i), see \cite[\S 2.2]{Humphreys} where it is explained that $\Phi_\theta^\vee$ is the root system of the centraliser $C_{G^\vee}(\theta)$. (Since we have assumed $G$ has connected centre, Steinberg's theorem implies that $C_{G^\vee}(\theta)$ is connected.) 
For (ii), see  \cite[Theorem 5.13]{DeligneLusztig}. 
\end{proof} 

\begin{rem} \label{r:subsystems}
As noted in Remark 5.14 of \cite{DeligneLusztig}, the closed subsystems $\Phi_\theta^\vee$ are of a special kind; namely, if $\Phi^\vee$ is irreducible, then $\Phi_\theta^\vee$ is generated by a subset of simple coroots together with the negative of the highest coroot. We will not use this fact. 
\end{rem}

\subsubsection{} Let $\ell$ and $\ell_\theta$ denote the length functions on $W$ and $W_\theta$, respectively. Let 
\[
P(t):= \sum_{w\in W} t^{\ell(w)} \qquad \textrm{and} \qquad P_\theta(t):=\sum_{w\in W_\theta} t^{\ell_{\theta}(w)}
\]
denote the corresponding Poincar\'e polynomials.

\begin{prop} \label{p:chi_theta} 
\begin{enumerate} 
\item[(i)] $\displaystyle \chi_\theta(1)=\frac{P(q)}{P_\theta(q)}$. Thus $\displaystyle \frac{|G(k)|}{\chi_\theta(1)}=P_\theta(q)\cdot \frac{|G(k)|}{P(q)}=P_\theta(q)\cdot |B(k)|$. 
\item[(ii)] $\chi_\theta(N)=1$. 
\item[(iii)] $\chi_\theta= \chi_{\theta'}$ if and only if $\theta$ and $\theta'$ are $W$-conjugate. 
\item[(iv)] If $S\in T(k)$ is strongly regular then 
$\displaystyle \chi_\theta(S)=\frac{1}{|W_\theta|} \sum_{w\in W} \theta(w.S)$.  
\end{enumerate} 
\end{prop} 
\begin{proof} 
\begin{enumerate} 
\item[(i)] This is proved in \cite[\S 5.11]{Kilmoyer} for semisimple $G$ of adjoint type. The proof for reductive $G$ with connected centre follows from the combination of  \cite[Theorem 2.6.11 (b)]{geck2020character}, \cite[Proposition 3.5.1, Remark 11.2.2]{digne2020representations} and \cite[Theorem 1.6.7]{geck2020character}. 

\item[(ii)] The  induced character formula implies that the character of $\cB(\theta)$ at $N$ is positive. Further by Theorem \ref{t:GLL} $\chi_\theta$ is the \emph{unique} constituent of $\cB(\theta)$ whose character does not vanish at $N$, and its value at $N$ is $\pm 1.$ We conclude $\chi_\theta(N)=1$.

 \item[(iii)]  This follows from Lemma \ref{l:principalSeries} because if $\chi_\theta=\chi_{\theta'}$ then $\cB(\theta)$ and $\cB(\theta')$ have a common irreducible constituent. 
 
\item[(iv)] Theorem \ref{t:DL} implies 
\[
\chi_\theta(S) = \sum_{\theta'\in T(\Fq)^\vee} \langle \chi_\theta, \cB(\theta')\rangle \theta'(S).
\]
By the above discussion
\[
\langle \chi_\theta, \cB(\theta')\rangle = 
\begin{cases} 
1 & \textrm{$\theta'=w.\theta$ for some $w\in W$;}\\
0 & \textrm{otherwise}.
\end{cases} 
\]
Thus
\[
\chi_\theta(S) = \sum_{\theta'\in W.\theta} \theta'(S)= \frac{1}{|W_\theta|}
\sum_{w\in W} (w.\theta)(S) = \frac{1}{|W_\theta|} \sum_{w\in W} \theta(w.S). 
\]
\end{enumerate} 
\end{proof}

\subsection{Point count} 
In this subsection, we begin counting points on the character variety  with regular monodromy $\bX=\bX(C_1,...,C_n)$ (Definition \ref{d:main}).  Recall that $C_1,...,C_m$ are strongly regular semisimple and $C_{m+1}, ..., C_n$ are regular unipotent; thus,  
\[
|C_i(k)| = 
\begin{cases} 
|G(k)|\cdot |T(k)|^{-1} &  i={1,2,...,m} \\
|G(k)|\cdot |Z(k)|^{-1} q^{-r} & i={m+1,...,n}.
\end{cases} 
\]

\subsubsection{} By the Frobenius Mass Formula \eqref{eq:Frob}, we have 
\[
	|\bX(k)|=\mathfrak{Z}\sum_{\chi\in {\Irr(G(k))}} \left(\frac{|G(k)|}{\chi(1)}\right)^{2g+n-2} \prod_{i=1}^n {\chi(C_i(k))}, 
\]
where
\begin{equation} \label{eq:Z} 
\mathfrak{Z}\colonequals |Z(k)|\cdot \prod_{i=1}^n \frac{|C_i(k)|}{|G(k)|} = |Z(k)|\cdot  |T(k)|^{-m}\cdot  \Big(|Z(k)|q^{r}\Big)^{m-n}. 
\end{equation} 
Note that $\mathfrak{Z}$ is a rational function in $q$. 

\subsubsection{} \label{sss:RS}
Since at least one of the $C_i$'s is strongly regular, Theorem \ref{t:DL} implies that in the above sum, we only need to consider those $\chi\in {\Irr(G(k))}$ that are principal series representations appearing in $\cB(\theta)$ for some $\theta\in {T(k)}^\vee$. On the other hand, since at least one of the $C_i$'s is regular unipotent, Theorem \ref{t:GLL} implies that it suffices to consider the unique constituent $\chi=\chi_{\theta }$. Finally, in view of Proposition \ref{p:chi_theta}.(iii) two characters in the same Weyl orbit yield isomorphic  constituents. Therefore the Frobenius sum reduces to a sum over $T(k)^\vee/W$, which (in view of the fact that $\chi_\theta(N)=1$) equals:
	\[
		|\bX(k)|= 
		 \mathfrak{Z} 
		\sum_{\theta \in {T(k)^\vee}/W} 
		\left(\frac{|G(k)|}{\chi_\theta(1)}\right)^{2g+n-2} \prod_{i=1}^m \chi_\theta(S_i). 	
\]
For convenience, we re-write this as a sum over $T(k)^\vee$: 
\[
|\bX(k)|=\mathfrak{Z} 
		\sum_{\theta \in {T(k)^\vee}} \frac{|W_\theta|}{|W|}
		\left(\frac{|G(k)|}{\chi_\theta(1)}\right)^{2g+n-2} \prod_{i=1}^m \chi_\theta(S_i). 
\]

\subsubsection{} 
Next, Proposition \ref{p:chi_theta}.(i) implies that $\chi_\theta(1)$ depends only on the stabiliser $W_\theta$. By theorem \ref{t:Steinberg}, we have  $W_\theta=W(\Phi_\theta^\vee)$. Collecting the terms corresponding to the same closed subsystem $\Psi=\Phi_\theta^\vee\subseteq \Phi^\vee$ and writing $P_\Psi=P_\theta$ for the Poincar\'e polynomial of $W(\Phi_\theta^\vee)$ yields
\[
	|\bX(k)|= \mathfrak{Z} \sum_{\Psi \in \mathcal{S}(\Phi^\vee)} \frac{|W(\Psi)|}{|W|}\cdot \left(P_\Psi(q)\cdot |B(k)|\right)^{2g+n-2}  
	\left( \sum_{\substack{\theta\in T(\kf)^\vee\\ W_\theta=W(\Psi)}}\,\,  
	\prod_{i=1}^m  \chi_\theta(S_i) \right).
\]

\subsubsection{} 
Using Proposition \ref{p:chi_theta}.(iv), we can rewrite the inner sum as follows: 
\[
\sum_{\substack{\theta\in T(\kf)^\vee\\ W_\theta=W(\Psi)}}\,\,  
\prod_{i=1}^m  \chi_\theta(S_i) 
=\frac{1}{|W(\Psi)|^m} \sum_{\substack{\theta\in T(\kf)^\vee\\ W_\theta=W(\Psi)}}\,\,  
\prod_{i=1}^m  \left( \sum_{w\in W} \theta(w.S_i)\right).
\]

\subsubsection{}  Let $\underline{S}$ denote the tuple $(S_1,...,S_m)$. For each $\underline{w}=(w_1,...,w_m)\in W^m$, let $\underline{w}.\underline{S}$ denote the product $(w_1.S_1)...(w_m.S_m) \in T(k)$. 
Then we can rewrite the above sum as:  
\[
\sum_{\substack{\theta\in T(\kf)^\vee\\ W_\theta=W(\Psi)}}\,\,  
\prod_{i=1}^m  \left( \sum_{w\in W} \theta(w.S_i)\right) =
\sum_{\substack{\theta\in T(\kf)^\vee\\W_\theta=W(\Psi)}} \,\,\, \sum_{\underline{w}\in W^m} \theta(w_1.S_1)\cdots \theta(w_m.S_m)=
 \sum_{\underline{w}\in W^m} \,\,  \sum_{\substack{\theta\in T(\kf)^\vee\\W_\theta=W(\Psi)}} \theta(\underline{w}.\underline{S}).
\]

\subsubsection{} \label{sss:alpha} For ease of notation, let 
\[
\boxed{
\alpha_{\Psi,S}:=   \sum_{\substack{\theta\in T^\vee(\kf)\\W_\theta=W(\Psi)}} \theta(S).}
\]
Then we can summarise the results of this subsection in the following  proposition: 
\begin{prop} \label{p:count1}
We have 
\[
\displaystyle |\bX(k)| = \frac{\mathfrak{Z}.|B(k)|^{2g+n-2}}{|W|} \sum_{\Psi \in \mathcal{S}(\Phi^\vee)} |W(\Psi)|^{1-m}\cdot P_\Psi(q)^{2g+n-2}. \left(\sum_{\underline{w}\in W^m} \alpha_{\Psi, \underline{w}.\underline{S}} \right).
\] 
\end{prop} 

To understand $|\bX(k)|$, it remains to analyse the character sums  $\alpha_{\Psi,S}$. To this end, we first need a detour on fixed points of  Weyl groups acting on tori.

\section{Invariants of Weyl groups on tori} The goal of this section is to review some facts we need regarding invariants of Weyl groups on the maximal torus. These facts will be used in the next section for computing the character sums $\alpha_{\Psi, S}$ and proving Theorem \ref{t:count}. 
We  work with the notation of \S \ref{ss:notation}, except we do not assume that $Z(G)$ is connected. 

 Let $T^W$ be the functor which associates to every $k$-algebra $R$, the set of fixed points 
\[
T(R)^W:=\{t\in T(R)\, | \, w.t=t, \,\, \forall w\in W\}.
\]
In this section, we discuss the number of points $T^W$ over a finite field. The results are presumably well-known though we could not find a sufficient treatment in the literature.

\subsection{The structure of $T^W$} We start by recalling some facts about $T^W$, cf.  \cite[\S 3.2]{Springer}. 
Observe that $T^W$ is represented by a closed (but not necessarily connected) algebraic subgroup of $T$. 
Indeed, $T^W=\mathrm{Spec}\, k[X_W]$, 
where 
$X_W$ is the group of coinvariants  of $W$ on  $X$; i.e., $X_W:=X/D_W$, where 
\[
 D_W:=\langle x-w.x \, | \, x\in X, \, w\in W\rangle.
\]
Note that $T^W$ is smooth if and only if $p$ is not a torsion prime for $X/D_W$. For instance, if $G=\mathrm{SL}_p$, then $T^W\simeq \mu_p$ which is not smooth in characteristic $p$.

\subsubsection{} Next, 
let $\Tor(X_W)$ denote the torsion subgroup of $X_W$ and $F(X_W)$ its maximal free quotient. Then we have a short exact sequence of finitely generated abelian groups 
\[
0\ra \Tor(X_W)\ra X_W\ra F(X_W)\ra 0.
\]
Applying the $\mathrm{Spec}$ functor, we obtain a short exact sequence of groups of multiplicative type
\[
1\leftarrow \pi_0(T^W) \leftarrow  T^W \leftarrow (T^W)^\circ  \leftarrow 1. 
\]

\subsubsection{Neutral component} We have: 
\begin{lem}\label{lem:lattices-containment}
 \[
2\langle \Phi \rangle \subseteq D_W \subseteq \langle \Phi \rangle.
\]
\end{lem}
  \begin{proof}
Let $s_{\alpha }\in W$ denote the reflection associated to a root $\alpha\in \Phi.$ To see $2\langle \Phi \rangle\subseteq D_W$ note that $2\alpha =\alpha-s_{\alpha }\alpha \in D_W.$ To see $D_W\subseteq \langle \Phi \rangle$ it suffices to show $x-w.x\in \langle \Phi \rangle$ for any $x\in X$ and $w\in W.$ For $x\in X$ and $w=s_{\beta_{\ell}}\cdots s_{\beta_1}\in W,$ set $x_0:=x$ and $x_{i+1}=s_{\beta_{i+1}}x_i$ for $0\leq i\leq \ell-1.$ Then 
$$x-w.x=x_0-x_{\ell}=\sum_{i=0}^{\ell-1}(x_i-x_{i+1})=\sum_{i=0}^{\ell-1}(x_i-s_{\beta _{i+1}}x_i)=\sum_{i=0}^{\ell-1}\langle x_i, \beta_{i+1}^{\vee}\rangle \beta_{i+1}\in  \langle \Phi \rangle. \eqno \qedhere$$ 
  \end{proof}

Thus, the abelian groups $X/D_W$ and  $X/\langle \Phi \rangle$ differ only in $2$-torsion; in particular,  the ranks of these groups are equal. Recall that the centre of $G$ is given by  $Z:=\mathrm{Spec}\, k[X/\langle \Phi \rangle]$. Thus, we conclude 
\[
(T^W)^\circ=Z^\circ.
\] 

\begin{rem} 
In fact, one can show that 
$T^W \simeq (\mathbb{Z}/2)^{\mathfrak{r}}\times Z$,
where $\mathfrak{r}$ is the number of direct factors of $G$ isomorphic to $\mathrm{SO_{2n+1}}$, for some $n\geq 1$, cf. \cite[Proposition 3.2]{JMO}.  We shall not use this fact. 
\end{rem}

 \subsubsection{The group of components}  
 For ease of notation, let $\fT:=\Tor(X_W)$ and $\pi_0:=\pi_0(T^W)=\mathrm{Spec}\, k[\fT]$. Then 
\[
\pi_0(\overline{k}) = \Hom(\fT, \overline{k}^\times). 
\]
If $p$ does not divide $|\fT|$, then $\pi_0$ is \'etale and
\[
\pi_0(k) = \pi_0(\overline{k})^{\mathrm{Gal}(\overline{k}/k)} = \Hom(\fT, \overline{k}^\times)^{\mathrm{Gal}(\overline{k}/k)}.
\]

\subsubsection{Group of components over finite fields} Now suppose $k=\Fq$ and $\gcd(q, |\fT|)=1$. Then, we have 
\[
\pi_0(\overline{k}) = \Hom(\fT, \overline{k}^\times) = \Hom(\fT, \mu_{p'}) = \Hom(\fT,\bCt) = \fT^\vee. 
\]
Here, the second equality uses the identification \eqref{eq:rootsOfUnity} while the third one follows from the fact that $\gcd(q,|\fT|)=1$. Taking the Galois fixed points and noting that the Galois group is generated by the Frobenius $\mathrm{Fr}$ which raises elements to power of $q$,  we obtain
\[
\pi_0(k) = \pi_0(\overline{k})^{\mathrm{Gal}(\overline{k}/k)}= \pi_0(\overline{k})^{\mathrm{Fr}}= \pi_0(\overline{k})^{q}=(\fT^\vee)^q. 
\]

\subsubsection{Polynomial property} 
 Let us show that $T^W$ is polynomial count.

\begin{lem} \label{l:T^WPoly}
Suppose $q\equiv 1 \mod |\fT|$. Then $T^W$ is polynomial count with counting polynomial 
\[
|\!|T^W|\!|(t)=|\Tor(X_W)| (t-1)^{\mathrm{rank}(X/\langle \Phi \rangle)}. 
\]
\end{lem} 

\begin{proof} Indeed, if $q\equiv 1 \mod |\fT|$, then $(\fT^\vee)^q=\fT^\vee$. Thus, the above discussions imply
\[
|T^W(\Fq)|=|\fT|(q-1)^{\rank(X/\langle \Phi \rangle)}.
\]
The argument goes through if $q$ is replaced with $q^n$, establishing the lemma. 
\end{proof} 

\begin{rem} 
We note that if $q$ is co-prime to $|\fT|$, then after a finite base change it becomes congruent to $1$ modulo $|\fT|$. 
\end{rem} 

\subsection{Invariants for subgroups of $W$} 
\begin{lem} \label{l:D=Q} Suppose $X^\vee/\langle \Phi^\vee \rangle$ is free. Let $\Psi$ be a root subsystem of $\Phi$ with Weyl group $W(\Psi)$. Then $D_{W(\Psi)}=\langle \Psi\rangle$. Thus, $T^{W(\Psi)}=\mathrm{Spec}\, k[X/\langle \Psi\rangle]=Z(G(\Psi))$, where $G(\Psi)$ is the connected reductive subgroup of $G$ containing $T$ and with root system $\Psi$.  
\end{lem}

\begin{proof} This is stated without proof in \cite[p 1035]{Deriziotis85}. For completeness, we sketch a proof. 
The inclusion $D_{W(\Psi)}\subseteq \langle \Psi\rangle$ follows from the fact that if $x\in X$ and $\alpha \in \Psi$, then
\[
x-s_\alpha x = \langle x,\alpha^\vee \rangle \alpha \in \langle \Psi \rangle. 
\]

For the reverse inclusion, we need to show that every $\beta \in \Psi$ is in $D_{W(\Psi)}$. 
Let $\beta^\vee$ be the coroot corresponding to $\beta$ (under a fixed $W$-invariant bijection $\Phi\ra \Phi^\vee$, cf. \cite[\S 1.2]{geck2020character}).
Let $w\in W$ be such that $\alpha:=w.\beta$ is a simple root of $\Phi$ (cf. \cite[Theorem V.10.2.(c)]{Serre}). Let $\mu_\alpha\in X\otimes \mathbb{Q}$ be the corresponding fundamental weight. The assumption that $X^\vee/\langle \Phi^\vee \rangle$ is free implies that $\mu_\alpha \in X$. Now, we have 
\[
 \langle w^{-1}\mu_\alpha, \beta^\vee\rangle = \langle \mu_\alpha, w.\beta^\vee\rangle = \langle \mu_\alpha, \alpha^\vee\rangle  = 1. 
\]
Thus, 
\[
\beta = \langle w^{-1}\mu_\alpha, \beta^\vee\rangle \beta = w^{-1}\mu_\alpha - s_\beta (w^{-1}\mu_\alpha)\in D_{W(\Psi)}. 
\]
Thus, $D_{W(\Psi)}=\langle \Psi\rangle$ which implies that $X_W=X/\langle \Psi \rangle$, establishing the last statement of the lemma. 
\end{proof}

\subsection{Modulus of a reductive group} \label{s:modulus} 
   \begin{defe}  \label{d:modulus} 
The \emph{modulus of $G$}, denoted by $d(G)$, is defined to be the least common multiple of $|\mathrm{Tor}(X/\langle \Psi \rangle)|$, where $\Psi$ ranges over all closed subsystems of $\Phi$. 
 \end{defe} 
 
One checks that $d(\GL_n)=1$. On the other hand, suppose $G$ is (almost) simple and simply connected. Then it is shown in \cite{Deriziotis85}  that $d(G)$ equals the least common multiple of coefficients of the highest root and the order of $Z(G)$. Thus, we have:
  \[
  \begin{array}{|c|c|c|c|c|c|c|c|c|c|}\hline
\mathrm{Type} & A_n  & B_n  & C_n  & D_n  & E_6 & E_7 & E_8 & F_4 & G_2\\\hline 
d(G)  & n+1 & 2 & 2  & 4 & 6  &  12 & 60 & 12  & 6 \\\hline 
 \end{array}
\]
Note that for types $B_n$, $C_n$, $E_6$, $G_2$ (resp. $D_n$, $E_7$, $E_8$),  $d(G)$ is the product of bad primes (resp. twice the product of bad primes) of $G$. 

\begin{prop}  \label{p:q1}
Suppose $k=\Fq$, $q\equiv 1 \mod d(G)$, and  $X^\vee/\langle \Phi^\vee\rangle$ is free.  Then for every closed subsystem $\Psi\subseteq \Phi$, the variety $T^{W(\Psi)}$ is polynomial count with counting polynomial 
\[
|\!|T^{W(\Psi)}|\!| (t)=|\mathrm{Tor}(X/\langle \Psi \rangle)| (t-1)^{\mathrm{rank}(X/\langle \Psi \rangle )}.
\]
\end{prop}

\begin{proof} By Lemma \ref{l:D=Q}, the assumption on $X^\vee/\langle \Phi^\vee\rangle$ implies that  $X_{W(\Psi)}=X/\langle \Psi\rangle$. Next,  the fact that $q\equiv 1 \mod d(G)$ implies that $q\equiv 1 \mod |\Tor (X/\langle \Psi\rangle)|$. The result then follows from Lemma \ref{l:T^WPoly}. 
\end{proof}

\subsubsection{Example} Suppose $G=\GL_n$ and $\Psi$ is the root system of the Levi subgroup $L=\GL_{\lambda_1}\times \cdots \times \GL_{\lambda_r}$ where $\lambda_1\geq \lambda_2\geq \cdots \geq \lambda_r$ is a partition of $n$. Then 
\[
T^{W(\Psi)}=Z(L)\simeq \bGm^{r}.
\]
Thus, for all $q$, 
\[
|T^{W(\Psi)}(\Fq)|=(q-1)^r.
\]
For more examples, see \S \ref{s:Examples}.

   \section{Evaluation of character sums} \label{s:PartTwo}  
  The aim of this section is to compute the character sum $\alpha_{\Psi, S}$ and ultimately derive the point count for character varieties. The main subtlety here, which does not occur for $\GL_n$, is that character varieties associated with arbitrary reductive groups are not polynomial count on the nose, but become polynomial count after a finite base change. There are two reasons for this. First, the Weyl group invariants on the maximal torus may be disconnected. Second, for general groups, the inclusion $[G(k), G(k)] \subseteq [G, G](k)$ might be strict. In this section, we carefully address these issues and explain how an appropriate base change resolves them.   
   
    Recall our assumption that $G$ has connected centre and that $p:=\mathrm{char}(k)$ is a very good prime for $G$. This implies that for each $\theta\in T^\vee(k)$, the stabiliser  $W_\theta$ is a reflection subgroup of $W$; i.e., it is of the form $W(\Psi)$, where $\Psi$ is a closed subsystem of $\Phi^\vee$. 
    Let us recall the definition of the character sum $\alpha_{\Psi, S}$ introduced in \S \ref{sss:alpha}.
      \begin{defe} \label{d:alpha} 
For a closed subsystem $\Psi \subseteq \Phi^\vee$ and an element $S\in T(k)=T(\Fq)$,  let  
\[
\alpha_{\Psi,S}=\alpha_{\Psi, S}(q):=   \sum_{\substack{\theta\in T^\vee(\kf)\\W_\theta=W(\Psi)}} \theta(S). 
\]
\end{defe}

To evaluate this sum, we generalise the approach of Deriziotis \cite{Deriziotis85} who considered the case $S=1$. (The sum $\alpha_{\Psi,1}$ is closely related to the \emph{genus number} attached to $\Psi$, cf.  \cite{Deriziotis85}.) Note that in this section we do not need $S$ to be a strongly regular element. The following discussion is valid for
any $S\in T(k)$.

\subsection{An auxiliary sum} 
   To evaluate $\alpha_{\Psi,S}$, it is convenient to consider an auxiliary sum: 
   
   \begin{defe}\label{d:delta} For a closed subsystem $\Psi \subseteq \Phi^\vee$ and an element $S\in T(k)=T(\Fq)$, define 
 \[
\Delta_{\Psi,S}=\Delta_{\Psi,S}(q):=\displaystyle \sum_{\substack{\theta \in T^\vee(k)\\ W_\theta\supseteq  W(\Psi)}} \theta(S). 
\]
\end{defe} 
As we shall see $\alpha$ and $\Delta$ are related by a M\"{o}bius inversion. 

\subsubsection{} 
Observe that $W$ (and therefore $W(\Psi)$) acts on $T^\vee(k)$. Moreover, 
$W_\theta\supseteq  W(\Psi)$ if and only if $\theta$ is in the set of fixed points $(T^\vee(k))^{W(\Psi)}$. 
Thus, we can rewrite the sum as 
\begin{equation} \label{eq:fixed} 
\Delta_{\Psi,S}= \sum_{\theta \in (T^\vee(k))^{W(\Psi)}} \theta(S). 
\end{equation} 
To evaluate this, we need to recall a basic fact about Pontryagin duality. 

\subsubsection{} Given a homomorphism of finite abelian groups $f:A\ra B$, let $f^\vee: B^\vee\ra A^\vee$ denote the map $f^\vee(\tau) = \tau\circ f$. Note that if $f$ is surjective, then $f^\vee$ is injective.

\begin{lem} \label{l:Pontryagin} 
 Let  $f:A\twoheadrightarrow B$ be a surjective homomorphism of abelian groups. Then, for every $a\in A$,  we have 
\[
\sum_{\theta \in f^\vee(B^\vee)} \theta(a) =
\begin{cases} 
|B| & \text{if }f(a)=1\\
0 & \text{otherwise}. 
\end{cases} 
\]
\end{lem} 

\begin{proof} 
Indeed, 
\[
\displaystyle \sum_{\theta \in f^\vee(B^\vee)} \theta(a) =\sum_{\tau \in B^\vee} f^\vee(\tau)(a) =\sum_{\tau \in B^\vee} \tau(f(a) )=
\begin{cases} 
|B| & \text{if }f(a)=1\\
0 & \text{otherwise}. 
\end{cases}
\]
Here, the first equality follows from the injectivity of $f^\vee$, the second one is by definition of $f^\vee$,  and the third one is the elementary fact that sum of characters of $B$ evaluated an element $b\in B$ vanishes unless $b=1$, in which case the sum equals $|B|$. 
\end{proof}

\subsubsection{} We now return to evaluating $\Delta_{\Psi,S}$. By definition, we have an injective homomorphism of abelian groups 
\[
f_\Psi^\vee: (T^\vee(k))^{W(\Psi)} \hookrightarrow T^\vee(k). 
\]
By Pontryagin duality, we obtain a surjective homomorphism
\[
f_\Psi: T(k) \twoheadrightarrow \Big( (T^\vee(k))^{W(\Psi)}\Big)^\vee.
\]
Noe that this map depends on the ground field $\Fq$; thus, we sometimes denote it by $f_{\Psi,q}$. 
 Lemma \ref{l:Pontryagin} then implies: 

\begin{cor} \label{c:Delta} We have 
\[
\Delta_{\Psi,S}(q)=
\begin{cases} 
|(T^\vee(\Fq))^{W(\Psi)}| & \text{if } f_{\Psi,q}(S)=1\\
0 & \text{otherwise}.
\end{cases}
\]
\end{cor} 
\begin{proof} This follows from the previous lemma and the expression for $\Delta$ given in \eqref{eq:fixed}. 
\end{proof} 

To proceed further, we need to give a formula for $|(T^\vee(\Fq))^{W(\Psi)}|$ and elucidate the nature of the map $f_{\Psi, q}$. 

\subsection{Number of points of $T^\vee(k)^{W(\Psi)}$} \label{s:pointsTW}  
Let $G^\vee(\Psi)$ denote the connected reductive subgroup of $G^\vee$ with root system $\Psi\subseteq \Phi^\vee$. Since $G$ has connected centre, Lemma \ref{l:D=Q} implies 
\[
(T^\vee)^{W(\Psi)} = Z(G^\vee(\Psi)). 
\]
Next, let $\pi_0^\Psi$ denote the group of components of $Z(G^\vee(\Psi))(\overline{k})$. Since $p:=\mathrm{char}(k)$ is a very good prime for $G$ (and therefore also for $G^\vee)$, Theorem 1.1 of \cite{Herpel} implies that $Z(G^\vee(\Psi))$ is smooth; i.e., $p\nmid |\pi_0^\Psi|$. It follows that $q$ is co-prime to $ |\pi_0^\Psi|$. Thus, replacing $q$ by a finite power if necessary, we may assume that $q\equiv 1 \mod |\pi_0^\Psi|$.  In this case, Lemma \ref{l:T^WPoly} implies 
\[
|Z(G^\vee(\Psi))(\Fq)| = |\pi_0^\Psi| (q-1)^{r(\Psi)},
\]
where $r(\Psi)$ is the rank of the torus $Z(G^\vee(\Psi))$.

\subsection{On the map $f_\Psi$} In this subsection, we give a more direct description of the map $f_\Psi=f_{\Psi,q}$. 
Let $H_\Psi^\vee$ be the connected reductive subgroup of $G^\vee$ with maximal torus $T^\vee$ and root system  $\Psi$. Let $H=H_\Psi$ be the Langlands dual group over $k$. Thus, $H_\Psi$ is a connected split reductive group over $k$ with maximal split torus $T$ and root datum $(X, \Psi^\vee, X^\vee, \Psi)$. (Note that $H_\Psi$ is not necessarily a subgroup of $G$. If $H^\vee$ is a centraliser of a semisimple element, then $H_\Psi$ is an endoscopy group for $G$.)

\begin{prop}
	\label{p:f_Psi} 
	
	The map $f_{\Psi,q}: T(k)\ra \Big( (T(k)^\vee)^{W(\Psi)}\Big)^\vee$ coincides with the  quotient map 
\[
T(k)\ra T(k)/(T(k)\cap [H_\Psi(k), H_\Psi(k)]).
\]
\end{prop} 

\begin{proof} We have 
\begin{multline*}
\Big( (T(k)^\vee)^{W(\Psi)}\Big)^\vee = T(k)_{W(\Psi)} = T(k)/\langle (w.t)t^{-1} \, | \, w\in W, t\in T(k)\rangle\\ =
T(k)/(D_W\otimes k^\times) = 
T(k)/(\langle \Psi\rangle\otimes k^\times) = T(k)/(T(k)\cap [H(k), H(k)]). 
\end{multline*}
Here, the first equality follows from the fact that the dual of invariants equals co-invariants of the dual, the second is by definition, the third follows from the fact that $X^\vee\otimes k^\times = T(k)$, the fourth from Lemma \ref{l:D=Q} and the last from
 \cite{kamgarpour2012ramified}, Proposition 23. 
\end{proof}

\subsubsection{} The inclusion $[H_\Psi(k), H_\Psi(k)]\hookrightarrow [H_\Psi,H_\Psi](k)$ may be proper.  (For instance, the commutator subgroup of  $\PGL_n(k)$ is a proper subgroup equal to the kernel of the determinant map.) However, if $S\in [H_\Psi,H_\Psi](k)$, then there exists a finite extension $k'/k$ such that $S\in [H_\Psi(k'), H_\Psi(k')]$.\footnote{This is an indication that the commutator subgroup of an algebraic group is not well-behaved; see \cite{Masoud} for further details.}

\subsection{Polynomial property of $\Delta$ and $\alpha$} \label{s:BaseChange}
Let $g_\Psi: T\ra T/T\cap [H_\Psi,H_\Psi]$ denote the canonical quotient map . The advantage of $g_\Psi$ over $f_{\Psi, q}$ is that the former does not involve the $\Fq$-points of $H_\Psi$. 

\begin{prop} 
Fix $S\in T(k)$. There exists a positive integer $\ell$ such that for all closed subsystems $\Psi\subseteq \Phi^\vee$ and all positive integers $n$, the following holds: 
\[
\Delta_{\Psi,S}(q^{n\ell})=
\begin{cases} 
|\pi_0^\Psi| (q^{n\ell}-1)^{r(\Psi)} & \text{if  } g_\Psi(S)=1;\\
0 & \text{otherwise}.
\end{cases}
\]
\end{prop}

\begin{proof} There are finitely many closed subsystems $\Psi\subseteq \Phi^\vee$. Thus, we can choose $\ell$ so that the following properties hold for all such $\Psi$ and all positive integers $n$: 
\begin{enumerate} 
\item If  $S\in [H_\Psi, H_\Psi]$, then $S\in [H_\Psi(\mathbb{F}_{q^{n\ell}}), H_\Psi(\mathbb{F}_{q^{n\ell}})]$. 
\item We have $q^{n\ell} \equiv 1 \mod |\pi_0^\Psi|$. 
\end{enumerate} 
The first property implies that $g_\Psi(S)=1$ if and only if $f_{q^{n\ell}, \Psi}(S)=1$ for all $n$. The second condition guarantees that 
\[
|(T^\vee)^{W(\Psi)}( q^{n\ell})|= |\pi_0^\Psi| (q^{n\ell}-1)^{r(\Psi)},
\]
for all $n$.
These observations together with Corollary \ref{c:Delta} implies the result. 
\end{proof}

\subsection{Evaluating $\alpha_{\Psi,S}$ via M\"obius inversion}\label{ss:Mobius}   Let 
$\mu:\mathcal{S}(\Phi^\vee)\times \mathcal{S}(\Phi^\vee) \ra \mathbb{Z}$
denote the M\"obius function on the poset $\mathcal{S}(\Phi^\vee)$ of closed subsystems of $\Phi^\vee$. It follows from the definition that 
\[
\Delta_{\Psi, S} =  \sum_{\substack{\Psi' \in \cS(\Phi^\vee)\\ \Psi'\supseteq  \Psi}} \alpha_{\Psi', S}. 
\]
M\"obius inversion formula then implies 
\[
\alpha_{\Psi,S} = \displaystyle \sum_{\substack{\Psi' \in \cS(\Phi^\vee)\\ \Psi'\supseteq  \Psi}} \mu(\Psi,\Psi') \Delta_{\Psi',S}.
\]
 
 The previous proposition implies: 
 \begin{cor}\label{c:Base} There exists a polynomial $\gamma_{\Psi,S}\in \mathbb{Z}[t]$ and a positive integer $\ell$ such that 
\[
\alpha_{\Psi, S}(q^{n\ell}) = \gamma_{\Psi}(q^{n\ell})\qquad \forall n\geq 1. 
\]
\end{cor} 

\begin{proof} Indeed, the previous proposition implies that such a polynomial exists for $\Delta$. The corollary then follows from the fact that $\alpha$ can be written in terms of $\Delta$ using M\"{o}bius inversion on the poset of closed subsystems of $\Phi^\vee$, which is independent of the ground field. 
\end{proof} 
 
 \subsubsection{} In type $A_{n-1}$, the poset $\mathcal{S}(\Phi^\vee)$ is isomorphic to the poset of set-partitions of $\{1,2,...,n\}$, ordered by refinement. The M\"{o}bius function of the latter poset is given explicitly in, e.g., \cite[\S 3]{Stanley}. 
Namely, suppose  $\Phi_1\subseteq \Phi_2$ are two closed subsystems of $\Phi^\vee$ corresponding to partitions $\pi_1\leq \pi_2$. Let $t$ and $s$ be the number of blocks of $\pi_1 $ and $\pi_2$ respectively. Then 
  \[
  \mu(\Phi_1, \Phi_2)=(-1)^{t-s}\prod_{i=1}^s (t_i-1)!
  \]
    where $t_i$ is the number of blocks of $\pi_1$ whose union is the $i$-th block of $\pi_2$ for $i=1,2, \dots ,s$. For a description of the M\"{o}bius function in other root systems, see \cite{Deriziotis, FJ1, FJ2}.

\subsubsection{Consistency check} Suppose $S=1$ and $\Psi=\emptyset$. Then $\alpha_{\Psi, S}$ is the number of regular elements of $T^\vee(\Fq)$. Thus, if $G=\GL_n$, then  $\alpha_{\emptyset, 1}=(q-1)(q-2)\cdots (q-n)$. By the above discussions, we also have 
\[
\alpha_{\emptyset, 1} = \sum_{\Psi \in \mathcal{S}(A_{n-1})} \mu(\emptyset, \Psi) |T(\Fq)^{W(\Psi)}|. 
\]
Using the description of $\mu$ given in the previous paragraph, it follows that the right-hand side does indeed equal $(q-1)\cdots (q-n)$.

\subsection{Conclusion of the point count} \label{s:countPrecise}
Recall that Proposition \ref{p:count1} states  
\[
\displaystyle |\bX(\Fq)| = \frac{\mathfrak{Z}(q).|B(\Fq)|^{2g+n-2}}{|W|} \sum_{\Psi \in \cS(\Phi^\vee)} |W(\Psi)|^{1-m}\cdot P_\Psi(q)^{2g+n-2}. \sum_{\underline{w}\in W^m} 
\alpha_{\Psi, w.S}(q). 
\]
Here, $\mathfrak{Z}$ is a rational function in $q$. We also know that $|B(\Fq)|$ is a polynomial in $q$. For convenience, we denote this polynomial by $P_B(q)$.  
By the previous corollary, there exists a positive integer $\ell$ and polynomials $\gamma_{\Psi, w.S}$ such that 
\[
\alpha_{\Psi, w.S}(q^{n\ell}) = \gamma_{\Psi, w.S}(q^{n\ell}), \quad \qquad \forall \, n\geq 1,\,\, w\in W, \,\, \Psi\in \mathcal{S}(\Phi^\vee). 
\]
Now define a rational function $F_\bX(t)\in \mathbb{Z}(t)$ by  
\[
F_\bX(t):=\frac{\mathfrak{Z}(t).(P_B(t))^{2g+n-2}}{|W|} \sum_{\Psi \in \cS(\Phi^\vee)} |W(\Psi)|^{1-m}\cdot P_\Psi(t)^{2g+n-2}. \sum_{\underline{w}\in W^m} 
\gamma_{\Psi, w.S}(t). 
\]
By the above discussions, we have 
\[
|\bX(\mathbb{F}_{q^{n\ell}})| = F_{\bX}(q^{n\ell}), \quad \forall\,  n\geq 1. 
\]
Thus, the rational function $F_\bX$ takes on an integer for infinitely many values in $\mathbb{Z}$. This implies that $F_\bX$ is actually a polynomial; i.e., $F_\bX\in \mathbb{Z}[t]$, cf. \cite[Remark 2.7]{LetellierRV}. We thus conclude that $\bX\otimes_{\Fq} \mathbb{F}_{q^\ell}$ is polynomial count with counting polynomial $F_\bX$.

\subsubsection{} \label{s:Torus} 
Let us do a consistency check when $G=T$ is a torus. In this case, every element $g\in G$ is regular semisimple and the only unipotent element is $1$. Moreover, every conjugacy class is a singleton. Thus, the representation variety and the character variety coincide and we have 
 \[
\bR(C_1,...,C_n)=\bX(C_1,...,C_n) = T^{2g}.
\] 
On the other hand, since $\Phi^\vee=\emptyset$, the above formula gives 
\[
|\bX(k)| =\mathfrak{Z}|T(k)|^{2g+n-2} \Delta_{\emptyset, S} =  |T(k)|^{1-n}|T(k)|^{2g+n-2}|T(k)| = |T(k)|^{2g}. 
\]
Note that in this case no base change is necessary.

\section{Topological invariants of character varieties} \label{s:topology} 
The goal of this section is to prove Theorem \ref{t:topology}.

\subsection{Proof of Part (i)}\label{ss:Thm1pii}
We have already seen that $\bX$ is equidimensional. (This is part (ii) of Theorem \ref{t:Smooth}.) We have also seen that $\bX$ is, after a finite base change, polynomial count. It remains to show that the leading coefficient of the counting polynomial $|\!|\bX|\!|$ is $|\pi_1([G,G])|$.  

\subsubsection{} In view of discussions in \ref{s:countPrecise}, to understand the leading coefficient of $|\!|\bX|\!|$, we need to study the polynomials 
\[
Q_{\Psi,\Psi',S}:= P_\Psi(q)^{2g+n-2} \Delta_{\Psi', S}.
\]
 Here, $\Psi\subseteq \Psi'$ are closed subsystems of $\Phi^\vee$ and $S=S_1\cdots S_m$. 
By Corollary \ref{c:Delta}, we have
\[
\deg(Q_{\Psi,\Psi',S})\leq  (2g+n-2)|\Psi^+| + \dim((T^\vee)^{W(\Psi')}). 
\]
Thus, to prove our result, it is enough to show: 
  \begin{enumerate}
  \item   the leading coefficient of $Q_{\Phi^\vee,\Phi^\vee,S}$ is $|\mathrm{Tor}(X^\vee/Q^\vee)|$ and
 \item  $\deg(Q_{\Psi,\Psi',S})$ is maximal if and only if $\Psi=\Psi'=\Phi^\vee$.  
 \end{enumerate} 
 These are established in the two propositions below.

 \begin{prop} \label{p:extension} 
 \begin{enumerate} 
 \item[(a)] $\deg(Q_{\Phi^\vee,\Phi^\vee,S}) =  (2g+n-2)|\Phi^+| + \dim((T^\vee)^{W})$. 
 \item[(b)] The leading coefficient of $Q_{\Phi^\vee,\Phi^\vee,S}$ is $ |\mathrm{Tor}(X^\vee/Q^\vee)|$. 
 \end{enumerate} 
 \end{prop}  
 
 \begin{proof} For (a) it is enough to show that 
\[
 \Delta_{\Phi^\vee,S}=\displaystyle \sum_{\substack{\theta \in T(k)^\vee\\ W_\theta=W}} \theta(S)\neq  0.
 \]
  Since $X/\langle \Phi \rangle$ is free, we can apply \cite[Theorem 3.(iii)]{kamgarpour2012ramified} (or Proposition \ref{p:f_Psi}) to conclude that every $W$-invariant character $\theta\in T(k)^\vee$ extends to a character of $G(k)$. Thus 
\[ 
\theta\Big(T(k)\cap [G(k), G(k)]\Big)=1.
\]
On the other hand, we have assumed (see \S \ref{s:emptiness})  that $S\in [G(k), G(k)]$. It follows that $\theta(S)=1$ for all $W$-invariant $\theta$'s. Thus, $\Delta_{\Phi^\vee,S}=|(T^\vee(k))^W|$, establishing (a).

For (b), note that the leading coefficient equals $|\pi_0((T^\vee)^W)|$. Proposition \ref{p:q1} (applied to $G^\vee$ instead of $G$) implies that $|\pi_0((T^\vee)^W)| = |\mathrm{Tor}(X^\vee/Q^\vee)|$. 
 \end{proof}

 \begin{prop} $\deg (Q_{\Psi,\Psi',S})$ is maximal if and only if $\Psi=\Psi'=\Phi^\vee$. 
 \end{prop}

\begin{proof} We need to show that if $\Psi$ is not equal to $\Phi^\vee$, then 
	\[
	(2g+n-2)|\Phi^+| + \dim ((T^\vee)^W) > (2g+n-2)|\Psi^+| + \dim ((T^\vee)^{W(\Psi')}). 
	\]
This is equivalent to 
	\[
(2g+n-2)(|\Phi^+|-|\Psi^+|)  > \dim ( (T^\vee)^{W(\Psi')}) - \dim ((T^\vee)^W).
\]
Note that by assumption, 
\[
2g+n-2\geq 2,\qquad |\Phi^+|-|\Psi^+|>0,\qquad  \textrm{and}\qquad  (T^\vee)^{W(\Psi)}\supseteq (T^\vee)^{W(\Psi')}.
\]
 Thus it is enough to show that 
	\[
	2(|\Phi^+|-|\Psi^+|)=|\Phi|-|\Psi| > \dim ((T^\vee)^{W(\Psi)}) - \dim ((T^\vee)^W). 
	\]
We now reduce this to an inequality about irreducible root systems. Let 
\[
\Phi=\Phi_1\sqcup \cdots \sqcup \Phi_t, 
\]
 where $\Phi_i$ are irreducible root systems. Let 
 \[
  \Psi_i=\Phi_i\cap \Psi,\qquad 1\leq i\leq t.
  \] 
  Since $\Psi$ is a proper subsystem of $\Phi$, we may assume without the loss of generality that there exists $s\in \{1,2,...,t-1\}$ such that we have 
   \[
   \Psi_j=\Phi_j, \quad 1 \leq j\leq s, \qquad \textrm{and},\qquad  \Psi_i\subsetneq \Phi_i,\quad s+1\leq i\leq t.
   \] 
   We thus obtain: 
\[
|\Phi|-|\Psi|=\sum_{i=s+1}^t (|\Phi_i |- |\Psi_i|)
\]
and 
	\[
	\dim ((T^\vee)^{W(\Psi)})-\dim ((T^\vee)^W)=\sum_{i=s+1}^{t}(\rank(\Phi_i)-\rank(\Psi_i))\leq \sum_{i=s+1}^{t} \rank(\Phi_i).
	\]
Therefore, to prove our result, it suffices to show that for each irreducible root system $\Phi_i$ and proper subsystem $\Psi_i\subsetneq \Phi_i$, we have $ \rank(\Phi_i) < |\Phi_i|-|\Psi_i|$. This follows from the next lemma. 
\end{proof}

\begin{lem}\label{l:rootslemma} Let $\Phi$ be an irreducible root system of rank $r$ and $\Psi$  a  proper subsystem. Then 
\[
|\Phi|-|\Psi|\geq 2r.
\] 
\end{lem} 

\begin{proof} This is presumably well-known but we could not find a reference so we provide a proof by induction on $r$. The statement is obvious for $r=1$. For $r>1$ consider a base $\Delta\subset \Phi,$ and let $\alpha$ be an element of $\Delta$ not in $\Psi$.  Consider the root subsystem $\Phi' \subsetneq \Phi $ with base $\Delta\setminus \{\alpha \}.$ It has rank $r-1.$ 

Now $\Phi '$ is the union of (at most three) irreducible components whose ranks add up to $r-1$. Let $\Phi''$ be one such component and $s=\rank(\Phi'')$. 
We have two cases: 
\begin{enumerate} 
\item[(a)] If $\Phi'' \cap \Psi \neq \Phi''$ then by the inductive hypothesis $|\Phi''\setminus (\Phi'' \cap \Psi)|\geq 2s.$
\item[(b)] If $\Phi''\cap \Psi=\Phi''$, then $\Phi''$ contains $\beta_1,\ldots ,\beta_{s}$ so that $\pm s_{\beta_j}(\alpha )$, $1\leq j\leq s$ are $2s$ distinct elements of $\Phi\setminus \Psi\setminus \{\pm\alpha \}.$ Indeed, we can take each $\beta_j$ to be the sum of a few elements of $\Delta\cap \Phi'',$ so that one of them is connected to $\alpha $ in the Dynkin diagram of $\Phi $, and together they form a connected subset. Then $\alpha, \beta_1,\ldots ,\beta_{s}$ is a linearly independent set of roots, and the $\pm s_{\beta_j}(\alpha )$ are distinct from each other and $\pm\alpha$. 
\end{enumerate} 

In either case, we find $2s$ separate roots of $\Phi\setminus \Psi $ corresponding to each irreducible component $\Phi ''$ of $\Phi '.$ Together with $\pm\alpha $ these give $2r$ distinct elements of $\Phi\setminus \Psi .$ This completes the proof of the lemma. \end{proof} 

\begin{rem}
\begin{enumerate} 
\item  
The statement of Lemma \ref{l:rootslemma} is sharp. For instance, consider 
\[
B_2\supset A_1\times A_1,\qquad \textrm{or} \qquad A_r\supset A_{r-1}.
\]
\item 
The above discussions show that the degree of $|\!|\bX|\!|$ equals  $2g\dim(G)-2\dim([G,G])+\sum_{i=1}^n\dim(C_i)$. This is in agreement with Theorem \ref{t:Smooth}. 
\end{enumerate} 
\end{rem}

\subsection{Proof of Part (ii)} To prove that $\bX$ has Euler characteristic $0$, it is sufficient to show that $(q-1)$ divides the counting polynomial $F_\bX$ of \S \ref{s:countPrecise}. First, observe that $\mathfrak{Z}$, defined in \eqref{eq:Z}, is a rational function in $q$ and 
\[
\mathrm{ord}_{q-1}(\mathfrak{Z}) = (m-n+1)\dim(Z)-m\dim(T).
\] 
Next, $B(k)$ is a polynomial in $q$ with 
\[
\mathrm{ord}_{q-1}(B(k))=\dim(T). 
\]
Finally, by Corollary \ref{c:Delta}, 
\[
\mathrm{ord}_{q-1}(\Delta_{\Psi,S}) \geq \dim ((T^\vee)^{W(\Psi)}) \geq \dim ((T^\vee)^W) = \dim(Z(G^\vee))=\dim(Z(G)). 
\]
Combining these observations with the expression for $F_\bX$, one verifies that
\[ 
\mathrm{ord}_{q-1}(|\!|\bX|\!|) \geq (2g+n-2)\dim(T) + (m-n+1)\dim(Z)-m\dim(T)+\dim(Z). 
\]
Since, we have assumed $G$ is non-commutative, $\dim(T)>\dim(Z)$ and so when $g>0$ or $n-m>2$, we obtain: 
\[
\mathrm{ord}_{q-1}(|\!|\bX|\!|) \geq  (2g+n-m-2)\dim(T)-(n-m-2)\dim(Z)>0.\eqno\qed 
\]

\section{Examples} \label{s:Examples} 
We conclude the paper by providing some explicit examples. 

\subsection{The case of $\GL_2$ and $\GL_3$} \label{s:GL}

\subsubsection{The case $(G,g,n)=(\GL_2, g, 2)$ and $g\geq 0$}
Let  $C_1$ be the conjugacy class of $\diag(a,b)$, where $a\neq b$. If $ab=1 $ and $a,b\neq 1$, i.e., if the class is generic in the sense of \cite[\S 2]{HLRV}, then 
\[
|\!|\bX|\!|(q)=(q-1)^{4g-1}q^{2g-1}((q+1)^{2g}-1). 
\]

\subsubsection{The case $(G,g,n)=(\GL_3, g, 2)$ and $g\geq 0$}
Recall that the $A_2$ root system is
 \[
\Phi=\{\pm \alpha, \pm\beta, \pm(\alpha+\beta)\}
\]
and its poset of closed subsystems can be depicted as follows: 
\begin{center}
        \begin{tikzpicture}[scale=.5]
 \node (two) at (0,3) {$A_2$ };
  \node (a) at (-5,0) {$(A_1)_1$};
    \node (b) at (0,0) {$(A_1)_2$};
  \node (1) at (5,0) {$(A_1)_3$};
  \node (zero) at (0,-3) {$\emptyset $};
  \draw (zero) -- (a) --   (two);
   \draw (zero) -- (b) --   (two);
   \draw (zero) -- (1) --(two);
\end{tikzpicture}
\end{center}
where $(A_1)_1=\langle \alpha \rangle$, $ (A_1)_2=\langle  \beta \rangle$ and $(A_1)_3=\langle \alpha+\beta\rangle $. (Here, $\langle \cdot \rangle $ denotes the smallest root subsystem containing the roots listed within, e.g. $\langle \alpha \rangle =\{\pm \alpha \}$.)
The non-zero values of the M\"obius function for this poset are as follows: 
\begin{itemize}
\item $\mu(\Psi ,\Psi )=1$ for every cosed subsystem $\Psi $.
    \item $\mu((A_1)_i,A_2)= -1$, $\mu(\emptyset,(A_1)_i)=-1$ for $i=1,2,3$.
     \item $\mu(\emptyset,A_2)= 2$.
\end{itemize}

Let  $C_1$ be the conjugacy class of $\diag(a,b,c)$, where $a\neq b\neq c\neq a$. If $abc=1 $ and $1 \notin \{ab,bc,ac\}$, i.e., if the class is generic in the sense of \cite[\S 2]{HLRV}, the table below presents all the terms that appear in the counting polynomial: 
\begin{figure}[H]
\begin{tabular}{|c|c|c|c|c|c|c|c|}
\hline
$\cS(\Psi)$  & $X^\vee/\langle \Psi \rangle$& $\mathrm{Torsion}$ & $\mathrm{Rank}$ & $\Delta_{\Psi,S}$ & $\alpha_{\Psi,S}$ &$|W(\Psi)|$  &$P_\Psi(q)$\\
\hline
\makecell{$A_2$} & 
$\mathbb{Z} $&
$1$ & $1$& 
$q-1$ & 
$q-1$ &$6$ &$(q^2+q+1)(q+1)$\\
\hline
\makecell{$(A_1)_i$ for $i=1,2,3$} &$\mathbb{Z}\times \mathbb{Z}$&$1$&$2$& $q-1$ &$-(q-1)$&$2$&${q+1}$\\
\hline
\makecell{$\emptyset$} & $\mathbb{Z}\times \mathbb{Z}\times \mathbb{Z}$&$1$&$3$ &$q-1$&$2(q-1)$&$1$ & $1$\\
\hline
\end{tabular}
\end{figure}

Plugging these values into the counting polynomial (\S  \ref{s:countPrecise}), we find: \begin{equation*}
    \begin{split}
    |\!|\mathbf{X}|\!|(q)
    =  & {q^{6g-2}(q-1)^{6g-2}}\left( \underbrace{(q^2+q+1)^{2g}(q+1)^{2g}}_{\Psi=A_2}+\underbrace{(-3)\cdot (q+1)^{2g}}_{\Psi\in \{(A_1)_1,(A_1)_2,(A_1)_3\}}+\underbrace{2}_{\Psi=\emptyset}\right). 
\end{split}
\end{equation*}
On the other hand, if $abc=1$ and $ab=1$, then the table below presents all the terms that appear in the counting polynomial: 
\begin{figure}[H]
\begin{tabular}{|c|c|c|c|c|c|c|c|}
\hline
$\cS(\Psi)$  & $X^\vee/\langle \Psi \rangle$& $\mathrm{Torsion}$ & $\mathrm{Rank}$ & $\Delta_{\Psi,S}$ & $\alpha_{\Psi,S}$ &$|W(\Psi)|$  &$P_\Psi(q)$\\
\hline
\makecell{$A_2$} & 
$\mathbb{Z} $&
$1$ & $1$& 
$q-1$ & 
$q-1$ &$6$ &$(q^2+q+1)(q+1)$\\
\hline
\makecell{$(A_1)_1$} &$\mathbb{Z}\times \mathbb{Z}$&$1$&$2$& $q-1$ &$(q-1)(q-2)$&$2$&${q+1}$\\
\hline
\makecell{$(A_1)_i$ for $i=2,3$} &$\mathbb{Z}\times \mathbb{Z}$&$1$&$2$& $q-1$ &$-(q-1)$&$2$&${q+1}$\\
\hline
\makecell{$\emptyset$} & $\mathbb{Z}\times \mathbb{Z}\times \mathbb{Z}$&$1$&$3$ &$q-1$&$2(q-1)$&$1$ & $1$\\
\hline
\end{tabular}
\end{figure}
Plugging these values into the counting polynomial $F_\bX$, we find: \begin{equation*}
    \begin{split}
    |\!|\mathbf{X}|\!|(q)
    =  & {q^{6g-2}(q-1)^{6g-2}}\left( \underbrace{(q^2+q+1)^{2g}(q+1)^{2g}}_{\Psi=A_2}+\underbrace{(q+1)^{2g}(q-4)}_{\Psi\in \{(A_1)_1,(A_1)_2,(A_1)_3\}}+\underbrace{(-1)\cdot(q-3)}_{\Psi=\emptyset}\right). 
\end{split}
\end{equation*}
Note that we have the same computation when $bc=1$ or $ac=1.$

\subsubsection{The case $(G,g,n)=(\GL_2, 0,3)$} \label{s:subtle2} 
We have two cases: either we have two regular semisimple classes and one regular unipotent class or we have two regular unipotent classes and one regular semisimple class.

Let us first assume we have two regular semisimple classes. Specifically, 
let $C_1$ and $C_2$ be the conjugacy classes of $\diag(a,b)$ and $\diag(c,d)$, respectively, with $a\neq b$ and $c\neq d$. Recall, by the assumption of  \S \ref{s:emptiness},  $abcd=1$, for otherwise the character variety is empty. We then have two cases: 
\begin{enumerate} 
\item[(i)] 
If $1\notin\{ac, ad, bc, bd\}$ (i.e., the tuple $(C_1,C_2, C_3)$ is generic in the sense of \cite[\S2]{HLRV}), then one may check using our counting formula that ${\mathbf{X}}(C_1,C_2,C_3)$ is a singleton. 
We now give explicitly the unique element of $\bX$. 
To this end, we need to specify matrices $X_i\in C_i$, $i=1, 2, 3$,  satisfying $X_1X_2X_3=1$:  
\[
X_1=\begin{bmatrix} a & 0 \\ ab(c+d)-a-b & b\end{bmatrix},\quad
 X_2 = \begin{bmatrix} \frac{1}{a} & -\frac{1}{a} \\ -c-d+\frac{1}{a}+\frac{1}{b} & c+d-\frac{1}{a}\end{bmatrix}, \quad 
 X_3=\begin{bmatrix}1&1\\0&1\end{bmatrix}.
\] 

\item[(ii)] Suppose $1\in \{ac, ad, bc, bd\}$.  In this case, our counting formula implies that  $\bX$ has two points.  Let us assume that $c=a^{-1}$. Then $abcd=1$ implies that $d=b^{-1}$. We can then write down explicit expressions for the two points $(X_1, X_2,X_3)$ and $(Y_1,Y_2,Y_3)$ as follows:
\[
X_1=\begin{bmatrix} a & -a+b \\ 0 & b\end{bmatrix}, \quad 
X_2 = \begin{bmatrix} \frac{1}{a} & \frac{1}{b}-\frac{2}{a}  \\ 0 & \frac{1}{b}\end{bmatrix},\quad
 X_3=\begin{bmatrix}1&1\\0&1\end{bmatrix},
\]
and
\[
Y_1=\begin{bmatrix} b & a-b \\ 0 & a\end{bmatrix},\quad
 Y_2 = \begin{bmatrix} \frac{1}{b} & \frac{1}{a}-\frac{2}{b} \\ 0 & \frac{1}{a}\end{bmatrix},\quad 
 Y_3=\begin{bmatrix}1&1\\0&1\end{bmatrix}.
\]
Note that both representations are reducible. 
\end{enumerate}

 Next, suppose we have two regular unipotent and one regular semisimple monodromies. Again using our counting formula,  one sees that ${\mathbf{X}}(C_1,C_2,C_3)$ is a singleton. The unique element can be expressed as follows: 
\[
X_1= \begin{bmatrix}a+\frac{1}{a}  & \frac{a}{a^2-2a+1} \\ -a+2-\frac{1}{a}  & 0 \end{bmatrix},\quad
 X_2=\begin{bmatrix} 0 & -\frac{a}{a^2-2a+1} \\ 
a-2+\frac{1}{a}   &2\end{bmatrix},\quad
 X_3=\begin{bmatrix}1&1\\0&1\end{bmatrix}.
\]

\subsection{Rigid representations} \label{ss:rigid} Suppose $(G,g,n)=(\PGL_2, 0, 3)$ and $\mathrm{char}(k)\neq 2$. 
Let $C_1$ and $C_2$ be the conjugacy classes of $\diag(a,1)$ and $\diag(b,1)$, respectively. The requirements that these be strongly regular classes means that $a$ and $b$ are not equal to $\pm 1$. Let $C_3$ be the regular unipotent class. We assume that  $ab=t^2$ for some $t\in k^\times$, otherwise $\bX$ is empty. One can show that in the non-generic case, $|\bX|=3$. Assume we are in the generic case; i.e., $ab\neq 1\neq ab^{-1}$. Then one can check using Theorem \ref{t:count} that $|\bX(k)|=2$. We now write down explicit formulas for the two inequivalent representations $E_1=(X_1,X_2, X_3)$ and  $E_2=(Y_1, Y_2, Y_3)$: 
\[
X_1= \begin{bmatrix}
  a t&0 \\ (t-a)(t-1)& t
\end{bmatrix}, \qquad  
 X_2=\begin{bmatrix}
    -t&t\\(t-a)(t-1) & -t^2+t-a
\end{bmatrix}, \qquad 
X_3= \begin{bmatrix}
    1&1\\0&1
\end{bmatrix},
\]
and 
\[
Y_1= \begin{bmatrix}
   a t&0 \\ -(t+a)(t+1)&t
\end{bmatrix}, \qquad 
Y_2=\begin{bmatrix}
    -t&t\\-(t+a)(t+1) & t^2+t+a
\end{bmatrix},\qquad 
Y_3= \begin{bmatrix}
    1&1\\0&1
\end{bmatrix}.
\]
It is clear that $X_1$ and $Y_1$ are conjugate to $\diag(a,1)$. To see that $X_2$ is conjugate to $\diag(b,1)$ note that $\mathrm{tr}(X_2)=-t^2-a=-ab-a=-a(1+b)$ and $\det(X_2)=at^2=(-a)^2b$. Thus, trace and determinant of $X_2$ are the same as the trace and determinant of $-a. \diag(b,1)$. Hence, $X_2$ is conjugate to $\diag(b,1)$ in $\PGL_2$. Similarly, one sees that $Y_2$ is conjugate to $\diag(b,1)$. 

Suppose, towards contradiction, there exists $g\in \PGL_2$ such that $gX_ig^{-1}=Y_i$ for $i=1,2,3$. Such a $g$ would have to be in the centraliser of $X_3=Y_3$; i.e., we must have $g=\begin{bmatrix}
    1&x\\0&1
\end{bmatrix}$ for some $x\in k$. Now consider the equation $gX_1g^{-1}=Y_1$. Comparing the $(1,1)$ and $(2,2)$ entries of this matrix equation, we find: 
\begin{equation} 
x(t-a)(t-1)+at=cat \qquad \textrm{and}\qquad    t-x(t-a)(t-1)=ct,
\end{equation} 
for some $c\in k^\times$. This implies that $ a(ct-t)=t-ct$; i.e., $a=-1$, which is a contradiction. We conclude that $E_1$ and $E_2$ are not conjugate. 

The genericness assumption implies that the triples $E_1$ and $E_2$ are irreducible, but we may also check this directly as follows. We need to show that there is no proper parabolic subgroup $P\subseteq \PGL_2$ containing $X_1$, $X_2$, and $X_3$. But the only proper parabolic subgroup containing $X_3$ is the upper triangular matrices $B$ and this subgroup does not contain $X_1$ or $X_2$ because the generic assumption implies that $t-a$ and $t-1$ are non-zero.

\subsection{Character varieties with positive Euler characteristic}\label{s:subtle}
In this subsection, we give some examples of character varieties with regular monodromy and non-zero Euler characteristic. Recall that the Euler characteristic is obtained by substituting $q=1$ into the counting polynomial. It would be interesting to find a general expression for these Euler characteristics. 

\subsubsection{} Let $(G,g,m,n)= (\GL_2, 0, 3, 4)$,  with $C_1, C_2, C_3$ the conjugacy classes of $\mathrm{diag}(a_1,a_2)$, $\mathrm{diag}(b_1,b_2)$, $\mathrm{diag}(c_1,c_2)$, respectively and $C_4$ the regular unipotent conjugacy class. Let us assume that $a_1\neq a_2$, $b_1\neq b_2$ and $c_1 \neq c_2$ and $a_ib_jc_k\neq1 $ for all $(i,j,k)$; i.e., $(C_1,C_2, C_3, C_4)$ is generic in the sense of \cite[\S2]{HLRV}. Then 
our counting formula implies 
\[
|\bX(C_1,C_2,C_3,C_4)|=q(q+3).
\] 
Thus the Euler characteristic is $4$. 

\subsubsection{} Let  $(G,g,m,n)= (\GL_3, 0, 3, 4)$ with $C_1, C_2$, $C_3$ the conjugacy classes of $\mathrm{diag}(a_1,a_2,a_3)$, $\mathrm{diag}(b_1,b_2,b_3)$, $\mathrm{diag}(c_1,c_2,c_3)$, respectively and $C_4$ the regular unipotent conjugacy class. Let us assume that $a_i\neq a_j$, $b_i\neq b_j$ and $c_i \neq c_j$ for all $i \neq j$ and $a_ib_jc_k\neq1 $. Then our counting formula implies 
 \[
|\bX(C_1,C_2,C_3,C_4)|=q^4(q^4+6q^3+19q^2+42q+46).
\] 
Thus the Euler characteristic is $114$.

\subsubsection{}  Let $(G,g,m,n)= (\GL_2, 0, 2, 4)$  with $C_1$, $ C_2$ the conjugacy classes of $\mathrm{diag}(a_1,a_2)$, $\mathrm{diag}(b_1,b_2)$, respectively and $C_3,C_4$  the regular unipotent  class. Let us assume that $a_1\neq a_2$ and $b_1\neq b_2$  and $a_ib_j\neq1 $ for all $(i,j)$. Then 
\[
|\bX(C_1,C_2,C_3,C_4)|=q^2+2q-1.
\] 
Thus, the Euler characteristic is $2$. 

\subsubsection{} Finally, suppose  $(G,g,m,n)= (\GL_3, 0, 2, 4)$ with $C_1, C_2$ the conjugacy classes of $\mathrm{diag}(a_1,a_2,a_3)$, $\mathrm{diag}(b_1,b_2,b_3)$, respectively and $C_3, C_4$ the regular unipotent  class. Let us assume that $a_i\neq a_j$ and $b_i\neq b_j$ for all $i \neq j$ and $a_ib_j\neq1 $ for all $(i,j)$. Then 
 \[
 |\bX(C_1,C_2,C_3,C_4)|=q^2(q+1)^2(q^2+q+1)^2-9q^2(q+1)^2+12q^2.
 \] 
 Thus, the Euler characteristic is $12$.

\subsection{The case $G=\mathrm{SO}_5$}\label{ss:SO5} Recall that the $B_2=C_2$ root system is
 \[
\Phi^+=\{\alpha, \beta, \alpha+\beta, 2\alpha+\beta\}
\]
and its poset of closed subsystems can be depicted as follows: 
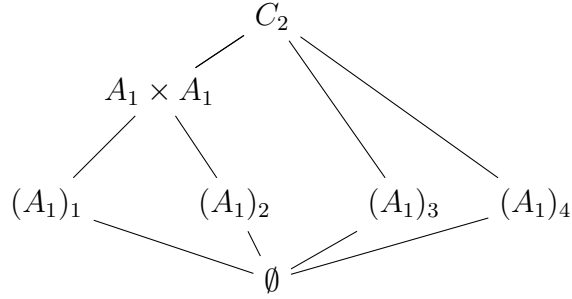
\begin{center}
        \begin{tikzpicture}[scale=.5]
 \node (two) at (0,3) {$C_2$ };
    \node (x) at (-3,1) {$ A_1\times A_1$};
  \node (a) at (-6,-2) {$(A_1)_1$};
    \node (b) at (-1,-2) {$(A_1)_2$};
  \node (1) at (3.5,-2) {$(A_1)_3$};
  \node (2) at (7,-2) {$(A_1)_4$};
  \node (zero) at (0,-4) {$\emptyset $};
  \draw (zero) -- (a) -- (x)--(two);
   \draw (zero) -- (b) -- (x)--(two);
   \draw (zero) -- (1) --(two);
   \draw (zero) -- (2)--(two);
\end{tikzpicture}
\end{center}

 \begin{itemize}
    \item $A_1 \times A_1=\langle \alpha,\alpha+2\beta\rangle$.
  \item $(A_1)_1=\langle \alpha \rangle$ and $ (A_1)_2=\langle \alpha+2\beta \rangle$.
    \item  $(A_1)_3=\langle \beta\rangle $ and $ (A_1)_4=\langle \alpha+\beta \rangle$. 
\end{itemize}

The non-zero values of the M\"obius function for this poset are as follows: 
\begin{itemize}
\item $\mu(\Psi ,\Psi )=1$ for every cosed subsystem $\Psi $.
    \item $\mu(A_1 \times A_1,C_2)= -1$.
    \item $\mu((A_1)_i,C_2)= 0$, $\mu((A_1)_i,A_1\times A_1)=-1$ for $i=1,2$.
     \item $\mu((A_1)_i,C_2)= -1$ for $i=3,4$.
     \item $\mu(\emptyset,C_2)=2$, $\mu(\emptyset, A_1\times A_1)=1$, and $\mu(\emptyset , (A_1)_i)=-1$,  for all $i=1,\ldots , 4$.
\end{itemize}

The table below presents all the terms that appear in the counting polynomial: 
\begin{figure}[H]
\begin{tabular}{|c|c|c|c|c|c|c|c|}
\hline
$\cS(\Psi)$  & $X^\vee/\langle \Psi \rangle$& $\mathrm{Torsion}$ & $\mathrm{Rank}$ & $\Delta_{\Psi,S}$ & $\alpha_{\Psi,S}$ &$|W(\Psi)|$  &$P_\Psi(q)$\\
\hline
\makecell{$C_2$} & 
$\mathbb{Z}/2\mathbb{Z}$&
$2$ & $0$& 
$2$ & 
$2$ &$8$ &$\frac{(q+1)(q^4-1)}{q-1}$\\
\hline
\makecell{$A_1\times A_1$} &$\mathbb{Z}/2\mathbb{Z}\times \mathbb{Z}/2\mathbb{Z}$&$4$& $0$ &$4$&$2$ &$4$&${(q+1)^2}$\\
\hline
\makecell{$(A_1)_1$ or $(A_1)_2$}  &$\mathbb{Z}\times \mathbb{Z}/2\mathbb{Z}$&$2$&$1$ &$0$& $-4$ &$2$& $q+1$ \\
\hline
\makecell{$(A_1)_3$ or $(A_1)_4$}  &$\mathbb{Z}$&$1$ &$1$  &$0$&$-2$ &$2$&$q+1$\\
\hline
\makecell{$\emptyset$} & $\mathbb{Z}\times \mathbb{Z}$&$1$&$2$ &$0$&$8$&$1$ & $1$\\
\hline
\end{tabular}
\end{figure}

Plugging these values into the counting polynomial, we find: \begin{equation*}
    \begin{split}
    |\mathbf{X}(k)|/\mathfrak{Z}
    =  & \underbrace{2 (q+1)^{2g+n-2} (q^3+q^2+q+1)^{2g+n-2}}_{\Psi=C_2} + \underbrace{2^{m} (q+1)^{4g+2n-4}}_{\Psi=A_1\times A_1}  \\
    & + \underbrace{ (-3)\cdot 2^{2m } (q+1)^{2g+n-2}}_{\Psi\in\{(A_1)_1,\ldots, (A_1)_4\}}+\underbrace{     8^m}_{\Psi=\emptyset},
\end{split}
\end{equation*}
where \[\mathfrak{Z}={(q-1)^{4g+2n-2m-4}q^{8g+2n+2m-8}}.\]

\subsection{The case $G=G_2$} \label{ss:G2} Recall that the $G_2$ root system is
 \[
\Phi^+=\{\alpha, \beta, \alpha+\beta, 2\alpha+\beta, 3\alpha+\beta, 3\alpha+2\beta\}
\]
and its poset of closed subsystems can be depicted as follows:
 
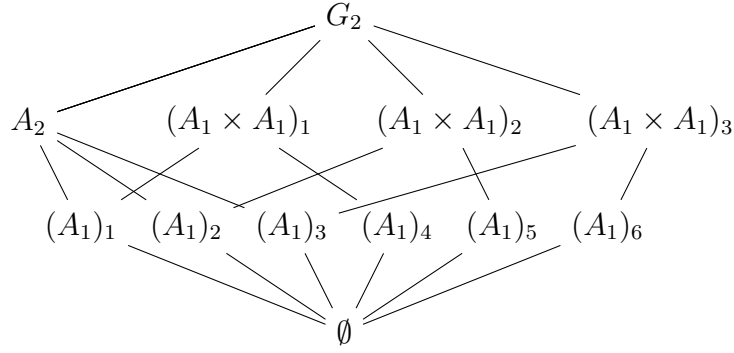
\begin{center}
        \begin{tikzpicture}[scale=.7]
  \node (two) at (0,2) {$G_2$ };
    \node (a) at (-6,0) {$A_2 $};
    \node (x) at (-2,0) {$(A_1 \times A_1)_1$};
    \node (y) at (2,0) {$(A_1 \times A_1)_2$};
    \node (z) at (6,0) {$(A_1 \times A_1)_3$};
  \node (b) at (-5,-2) {$(A_1)_1$};
  \node (c) at (-3,-2) {$(A_1)_2$};
  \node (d) at (-1,-2) {$(A_1)_3$};
    \node (1) at (1,-2) {$(A_1)_4$};
  \node (2) at (3,-2) {$(A_1)_5$};
  \node (3) at (5,-2) {$(A_1)_6$};
  \node (zero) at (0,-4) {$\emptyset$};
  \draw (zero) -- (d) -- (a)--(two);
    \draw (zero) -- (b) -- (a)--(two);
      \draw (zero) -- (c) --(a)-- (two);
      \draw (zero) -- (1) ;
 \draw (zero) -- (2) ;
  \draw (zero) -- (3) ;
  \draw (b) -- (x) --(1);
   \draw (c) -- (y) --(2);
    \draw (d) -- (z) --(3);
\draw (y)--(two)-- (x);
\draw(z)--(two);
\end{tikzpicture}
\end{center}

 \begin{itemize}
    \item $A_2=\langle \beta, 3\alpha+\beta\rangle$.
    \item $(A_1\times A_1)_1=\langle \alpha, 3\alpha+2\beta\rangle $, $(A_1\times A_1)_2=\langle \alpha+\beta, 3\alpha+\beta\rangle $ and $(A_1\times A_1)_3=\langle \beta, 2\alpha+\beta\rangle $.
  \item $(A_1)_1=\langle 3\alpha+2\beta  \rangle, (A_1)_2=\langle 3\alpha+\beta \rangle$ and $(A_1)_3=\langle \beta \rangle$.
    \item  $(A_1)_4=\langle \alpha  \rangle, (A_1)_5=\langle \alpha+\beta  \rangle$ and $(A_1)_6=\langle 2\alpha+\beta \rangle$. 
\end{itemize}

The non-zero values of the M\"obius function for this poset are as follows: 
\begin{itemize}
\item $\mu(\Psi ,\Psi )=1$ for every cosed subsystem $\Psi $.
    \item $\mu(A_2,G_2)= -1$.
    \item $\mu((A_1\times A_1)_i,G_2)=-1$ for all $i=1,2,3$.
    \item $\mu((A_1)_i,G_2)= 1$, $\mu((A_1)_i,A_2)=-1$, $\mu((A_1)_i,(A_1\times A_1)_i)=-1$ for $i=1,2,3$.
     \item $\mu((A_1)_i,G_2)= 0$, $\mu((A_1)_i,(A_1\times A_1)_{i-3})=-1$ for $i=4,5,6$.
     \item $\mu(\emptyset, G_2)=0$,  $\mu(\emptyset,A_2)=2$, $\mu(\emptyset, (A_1 \times A_1)_i)=1$, $\mu(\emptyset , (A_1)_j)=-1$ for all $i=1,2,3$ and $j=1,\ldots , 6$.
\end{itemize}

The table below presents all the terms that appear in the counting polynomial: 
\begin{figure}[H]
\begin{tabular}{|c|c|c|c|c|c|c|c|}
\hline
$\cS(\Psi)$  & $X^\vee/\langle \Psi \rangle$& $\mathrm{Torsion}$ & $\mathrm{Rank}$ & $\Delta_{\Psi,S}$ & $\alpha_{\Psi,S}$ &$|W(\Psi)|$  &$P_\Psi(q)$\\
\hline
\makecell{$G_2$}  & $1$&
$1$ & $0$& 
$1$ & 
$1$ &$12$ &$\frac{(q+1)(q^6-1)}{q-1}$\\
\hline
\makecell{$A_2$} & $\mathbb{Z}/3\mathbb{Z}$&$3$& $0$ &$3$&$2$ &$6$& $\frac{(q+1)(q^3-1)}{q-1}$\\
\hline
\makecell{$\substack{(A_1 \times A_1)_i\\ \text{for } i=1,2,3}$} & $\mathbb{Z}/2\mathbb{Z}$&$2$&$0$ &$2$& $1$ &$4$& $(q+1)^2$ \\
\hline
\makecell{$\substack{(A_1 )_i\\ \text{for } i=1,2,3}$}& $\mathbb{Z}$&$1$ &$1$  &$0$&$-4$ &$2$&$q+1$\\
\hline
\makecell{$\substack{(A_1 )_i\\ \text{for } i=4,5,6}$} &  $\mathbb{Z}$&$1$&$1$ &$0$&$-2$&$2$ & $q+1$\\
\hline
\makecell{$\emptyset$} &$\mathbb{Z}\times \mathbb{Z}$&$1$&$2$ &$0$&$12$&$1$&$1$\\
\hline
\end{tabular}
\end{figure}

Plugging these values into the counting polynomial, we find: 
\begin{equation*}
    \begin{split}
    |\mathbf{X}(k)|/\mathfrak{Z}
    =  & \underbrace{(q+1)^{2g+n-2} (q^5+q^4+q^3+q^2+q+1)^{2g+n-2}}_{\Psi=G_2} + \underbrace{2^{m} (q+1)^{2g+n-2}(q^2+q+1)^{2g+n-2}}_{\Psi=A_2}  \\
    &+  \underbrace{ 3^{m } (q+1)^{4g+2n-4}}_{\Psi\in \{ (A_1\times A_1)_1,\ldots,(A_1\times A_1)_3\}}+ \underbrace{  (-3) \cdot 6^{m} (q+1)^{2g+n-2} }_{\Psi\in \{ (A_1)_1,\ldots , (A_1)_6\}} + \underbrace{     12^m}_{\Psi=\emptyset},
\end{split}
\end{equation*}
where 
\[
\mathfrak{Z}={(q-1)^{4g+2n-2m-4}q^{12g+4n+2m-12}}.
\]


\end{document}